\documentclass[a4paper,12pt]{article}

\usepackage{amsthm}
\usepackage{amsmath}
\usepackage{amssymb}
\usepackage{amsfonts}
\usepackage{latexsym}
\usepackage{mathrsfs}
\usepackage{mathtools} 
\usepackage{xcolor}

\usepackage[T1]{fontenc}
\usepackage{lmodern} 
\usepackage{bm} 

\numberwithin{equation}{section}
\mathtoolsset{showonlyrefs=true}

\allowdisplaybreaks[1]

\usepackage{enumitem}

\renewcommand{\Re}{\operatorname{Re}}
\renewcommand{\Im}{\operatorname{Im}}

\DeclareMathOperator{\sgn}{sgn}

\newcommand{\ci}{\mathrm{i}}
\newcommand{\ce}{\mathrm{e}}
\newcommand{\cd}{\mathrm{d}}

\DeclarePairedDelimiter{\abs}{\lvert}{\rvert} 
\DeclarePairedDelimiter{\norm}{\lVert}{\rVert} 
\DeclarePairedDelimiter{\rbra}{(}{)} 
\DeclarePairedDelimiter{\cbra}{\{}{\}} 
\DeclarePairedDelimiter{\sbra}{[}{]} 

\newcommand{\bN}{\ensuremath{\mathbb{N}}}
\newcommand{\bO}{\ensuremath{\mathbb{O}}}
\newcommand{\bP}{\ensuremath{\mathbb{P}}}
\newcommand{\bQ}{\ensuremath{\mathbb{Q}}}
\newcommand{\bR}{\ensuremath{\mathbb{R}}}

\newcommand{\cD}{\ensuremath{\mathcal{D}}}

\newcommand{\cF}{\ensuremath{\mathcal{F}}}

\newcommand{\cH}{\ensuremath{\mathcal{H}}}

\newcommand{\cL}{\ensuremath{\mathcal{L}}}

\newcommand{\cP}{\ensuremath{\mathcal{P}}}

\theoremstyle{plain}
\newtheorem{Thm}{Theorem}[section]
\newtheorem{Lem}[Thm]{Lemma}
\newtheorem{Prop}[Thm]{Proposition}
\newtheorem{Cor}[Thm]{Corollary}

\theoremstyle{definition}

\newtheorem{Rem}[Thm]{Remark}
\theoremstyle{remark}
\newtheorem{Eg}[Thm]{Example}

\theoremstyle{plain}
\newtheorem*{Thm*}{Theorem}
\newtheorem*{Lem*}{Lemma}
\newtheorem*{Prop*}{Proposition}
\newtheorem*{Cor*}{Corollary}
\newtheorem*{Conj*}{Conjecture}
\theoremstyle{definition}
\newtheorem*{Ass*}{Assumption}
\newtheorem*{Def*}{Definition}
\newtheorem*{Rem*}{Remark}
\theoremstyle{remark}
\newtheorem*{Eg*}{Example}

\makeatletter
\renewcommand\section{\@startsection {section}{1}{\z@}%
                                   {-3.5ex \@plus -1ex \@minus -.2ex}%
                                   {2.3ex \@plus.2ex}%
                                   {\normalfont\large\bf}}
\makeatother

\makeatletter
\renewcommand\subsection{\@startsection {subsection}{1}{\z@}%
                                   {-3.5ex \@plus -1ex \@minus -.2ex}%
                                   {2.3ex \@plus.2ex}%
                                   {\normalfont\normalsize\bf}}
\makeatother

\usepackage[%
bookmarks=true,%
bookmarksdepth=3,%
bookmarksnumbered=true,%
setpagesize=false,%
pdftitle={Local time penalizations
with various clocks for L\'{e}vy processes},%
pdfauthor={Shosei Takeda and Kouji Yano},%
pdfkeywords={one-dimensional L\'{e}vy process; %
limit theorem; penalization; conditioning}%
]{hyperref}

\title{\textbf{Local time penalizations
with various clocks for L\'{e}vy processes}\footnote{
This research was supported by RIMS and by ISM.}}

\author{Shosei Takeda\footnote{Rakunan High School, Kyoto, Japan.}
\footnote{The research of this author was supported by
JSPS Open Partnership Joint Research Projects grant no. JPJSBP120209921.}
\quad and \quad Kouji Yano\footnote{Graduate School of Science, Kyoto University, Japan.}
\footnotemark[3]
\footnote{The research of this author was supported by
JSPS KAKENHI grant no.'s JP19H01791 and JP19K21834.}}
\date{}
\begin{document}
\maketitle
\begin{abstract}
  Several long-time limit theorems
  of one-dimensional L\'{e}vy processes
  weighted and normalized
  by functions of the local time
  are studied.
  The long-time limits are taken via certain families of random times, called clocks:
  exponential clock, hitting time clock, two-point hitting time clock
  and inverse local time clock.
  The limit measure can be characterized via a certain martingale
  expressed by an invariant function for the process
  killed upon hitting zero.
  The limit processes may differ according to the choice of the clocks
  when the original L\'{e}vy process is recurrent and of finite variance.
\end{abstract}
{\small Keywords and phrases: one-dimensional L\'{e}vy process;
limit theorem; penalization; conditioning }
\section{Introduction}
Roynette--Vallois--Yor
(\cite{MR2229621,MR2253307} see also~\cite{MR2261065,MR2504013})
have studied the limit distribution for a Brownian motion,
which they called a \textit{penalization problem},
as follows.
Let \(B=(B_t,t\ge 0)\) be a standard Brownian motion and
\(L=(L_t,t\ge 0)\) denote its local time at \(0\). Then, for
any positive integrable function \(f\)
and any bounded adapted functional \(F_t\), it holds that
\begin{align}
  \lim_{s\to\infty} \frac{\bP\sbra{F_t f(L_s)}}{\bP\sbra{f(L_s)}}
  = \bP\sbra*{F_t \frac{M_t}{M_0}} \eqqcolon \bQ\sbra{F_t},
\end{align}
where \(M=(M_t,t\ge 0)\) is the martingale given by
\begin{align}
  M_t = f(L_t)\abs{B_t} + \int_0^\infty f(L_t + u)\, \cd u,
  \quad t\ge 0.
\end{align}
Under the penalized probability
measure \(\bQ\), the total local time \(L_{\infty}\) is finite, and
in fact, a sample path behaves as the concatenation of a Brownian bridge
and a three-dimensional Bessel process; see~\cite{MR2229621}.
In particular, \(\bQ\) is singular to \(\bP\).

This result for a Brownian motion was
generalized to many other processes.
In particular, we refer to
Debs~\cite{MR2599215} for random walks,
Najnudel--Roynette--Yor~\cite{MR2528440} for Markov chains and Bessel processes,
Yano--Yano--Yor~\cite{MR2552915} for symmetric stable processes,
Salminen--Vallois~\cite{MR2540855} and Profeta~\cite{MR2760742,MR2968676} for linear
diffusions.
Most of these results were obtained basically under the assumption of some
regular variation condition.
Profeta--Yano--Yano~\cite{MR3909919} developed a general theory for
one-dimensional diffusions by adopting a random clock approach.
They studied the long-time limit of the form
\begin{align}
  \lim_{\tau\to\infty}\frac{\bP\sbra{F_t f(L_\tau)}}{\bP\sbra{f(L_\tau)}},
\end{align}
where \(\tau=(\tau_\lambda)\) is a certain parametrized
family of random times, which they called a clock.
Such a random clock approach already
appeared in the problem of conditioning to avoid zero,
which is a special case of our penalization with
\(f(u)=1_{\cbra{u=0}}\), or in the problem of conditioning
to stay positive/negative.
For example, we refer to Knight~\cite{MR253424} for Brownian motions,
Chaumont~\cite{MR1419491}, Chaumont--Doney~\cite{MR2164035,MR2375597}
and Doney~\cite[Section 8]{MR2126962,MR3155252}
for L\'{e}vy processes conditioned to stay positive,
Yano--Yano~\cite{MR3444297} for diffusions and
Pant\'{\i}~\cite{MR3689384} for L\'{e}vy processes conditioned to avoid zero.

Let \(X=(X_t,t\ge 0)\) be a one-dimensional L\'{e}vy process and
let \(T_A\) denote the hitting time of a Borel
set \(A \subset \bR\) for \(X\), i.e.,
\begin{align}\label{eq:T_A}
  T_A = \inf\cbra{t>0\colon X_t \in A}
\end{align}
and we write \(T_a = T_{\cbra{a}}\)
simply for the hitting time of a point \(a\in\bR\).
Let \((\eta^a_u)\) denote the right-continuous inverse of
the local time at a point \(a\in \bR\).
We adopt the random clock approach for the following four clocks:
\begin{enumerate}
  \item exponential clock: \(\tau=(\bm{e}_q)\) with \(q\to 0+\);
  \item hitting time clock: \(\tau=(T_a)\) with
        \(a\to\pm\infty\);\label{item:hitting-time-clocks}
  \item two-point hitting time clock: \(\tau=(T_{a}\wedge T_{-b})\) with
        \(a\to\infty\) and \(b\to\infty\);\label{item:two-point-hitting}
  \item inverse local time clock: \(\tau=(\eta^a_u)\)
        with \(a\to\pm\infty\) or with \(u\to\infty\).
\end{enumerate}

\subsection{Main results}\label{Subsec:main-results}
Let \((X, \bP_x)\) denote the canonical representation
of a L\'{e}vy process starting from \(x\)
on the c\`{a}dl\`{a}g path space \(\cD\)
and set \(\bP=\bP_0\).
For \(t\ge 0\), we denote by \(\cF_t^X = \sigma(X_s, 0\le s\le t)\) the natural
filtration of \(X\) and write \(\cF_t = \bigcap_{s>t} \cF_s^X\).
We have
\begin{align}
  \bP\sbra{\ce^{\ci\lambda X_t}} = \ce^{-t \varPsi(\lambda)},
  \quad t \ge 0, \,\lambda \in \bR,
\end{align}
where \(\varPsi(\lambda)\) denotes the characteristic exponent of \(X\)
given by the L\'{e}vy--Khintchine formula
\begin{align}
  \varPsi(\lambda)
  = \ci v \lambda
  + \frac{1}{2} \sigma^2 \lambda^2
  + \int_\bR \rbra*{1 - \ce^{\ci \lambda x} + \ci \lambda x 1_{\cbra{\abs{x}< 1}}}
  \nu(\cd x)
\end{align}
for some constants \(v \in \bR\) and \(\sigma \ge 0\)
and some measure \(\nu\) on \(\bR\)
(called the L\'{e}vy measure)
which satisfies \(\nu(\cbra{0})=0\) and
\begin{align}
  \int_\bR \rbra*{x^2 \wedge 1} \nu(\cd x) < \infty.
\end{align}
We denote the real and imaginary parts of \(\varPsi(\lambda)\) by
\begin{align}
  \theta(\lambda) & = \Re \varPsi(\lambda)
  = \frac{1}{2} \sigma^2 \lambda^2
  + \int_\bR \rbra*{1 - \cos \lambda x} \nu(\cd x),\label{eq:theta} \\
  \omega(\lambda) & = \Im \varPsi(\lambda)
  = v \lambda
  + \int_\bR \rbra*{\lambda x 1_{\cbra{\abs{x}< 1}} - \sin \lambda x} \nu(\cd
  x)\label{eq:omega}.
\end{align}
Note that \(\theta(\lambda) \ge 0\) for \(\lambda \in \bR\),
\(\theta(\lambda)\) is even and \(\omega(\lambda)\) is odd.
For more details of the notation of this section,
see Section~\ref{Sec:preliminaries}.
Throughout this paper except
Sections~\ref{Sec:preliminaries},~\ref{Sec:trans} and~\ref{Sec:mart},
we always assume
\((X,\bP)\) is recurrent, i.e.,
\begin{align}\label{eq:recurrent}
  \bP\sbra*{\int_0^\infty 1_{\cbra{\abs{X_t-a}<\varepsilon}}\, \cd t}
  =\infty, \qquad\text{for all \(a\in\bR\) and \(\varepsilon>0\),}
\end{align}
 and assume the following:
\begin{enumerate}[label=\textbf{(\Alph*)}]
  \item For each \(q > 0\), it holds that
        \begin{align}
          \int_0^\infty
          \abs*{\frac{1}{q+\varPsi(\lambda)}}
          \, \cd \lambda < \infty.
        \end{align}\label{item:assumption}
\end{enumerate}
Note that, we say \((X,\bP)\) is transient if~\eqref{eq:recurrent}
does not hold.
Under the assumption~\ref{item:assumption},
the process
\((X,\bP)\) is recurrent if and only if \((X,\bP)\) is
point recurrent, i.e.,
\begin{align}\label{eq:pt-recurrent}
  \bP(T_a<\infty) = 1,\qquad\text{for all \(a\in\bR\);}
\end{align}
see~Subsection~\ref{Subsec:LT-em}.
The assumption~\ref{item:assumption}
implies that the \(q\)-resolvent density \(r_q\) exists for \(q>0\);
see Subsection~\ref{Subsec:resolvent}.
For \( q > 0\), we define
\begin{align}\label{eq:def-h_q}
  h_q(x) = r_q(0) - r_q(-x)
  = \frac{1}{\pi}
  \int_0^\infty
  \Re\rbra*{\frac{1-\ce^{\ci \lambda x}}{q+\varPsi(\lambda)}}
  \,\cd \lambda
  , \qquad x \in \bR,
\end{align}
where the second identity follows from Proposition~\ref{Prop:r_q-represent}.
It is obvious that \(h_q(0) = 0\), and
by~\eqref{eq:exp-hit-time}, we have \(h_q(x) \ge 0\).
We denote the second moment by
\begin{align}
  m^2 =
  \bP\sbra{{X_1}^2} \in (0, \infty].\label{eq:second-moment}
\end{align}
The following theorem plays a key role
in our penalization results.
Recall that \(X\) is assumed recurrent.

\begin{Thm}\label{Thm:exist-h}
  Suppose that~\ref{item:assumption}
  is satisfied.
  Then the following assertions hold.
  \begin{enumerate}
    \item For any \(x \in \bR\),
          \begin{align}\label{eq:conv-h}
            h(x) \coloneqq \lim_{q \to 0+} h_q(x)
          \end{align}
          exists and is finite,
          which will be called the \emph{renormalized zero resolvent}.
          If \(m^2 < \infty\), then \(h\) has the following representation:
          \begin{align}\label{eq:h-repre}
            h(x)
            = \frac{1}{\pi} \int_0^\infty \Re\rbra*{
              \frac{1-\ce^{\ci \lambda x}}{\varPsi(\lambda)}}
            \, \cd \lambda.
          \end{align}\label{Thm-item:exist-h}
    \item The convergence~\eqref{eq:conv-h} is uniform on compacts, and
          consequently \(h\) is continuous.\label{Thm-item:unif-conv-h}
    \item \(h\) is subadditive on \bR,
          that is, \(h(x+y)\le h(x)+h(y) \) for \(x,y\in
          \bR\).\label{Thm-item:subadditive-h}
  \end{enumerate}
\end{Thm}
The proof of Theorem~\ref{Thm:exist-h} will be given in
Section~\ref{Subsec:pf-h}.
The renormalized zero resolvent satisfies the following limit properties.
\begin{Thm}\label{Thm:property-h}
  Suppose that~\ref{item:assumption}
  is satisfied.
  Then the following assertions hold:
  \begin{enumerate}
    \item
          \(\displaystyle
          \lim_{x\to\pm\infty} \frac{h(x)}{\abs{x}} =
          \frac{1}{m^2} \in [0,\infty)
          \);\label{Thm-item:h/x-infinity}
    \item
          \(\displaystyle\label{eq:h-h-limit}
          \lim_{y\to\pm\infty} \cbra*{h(x+y) - h(y)} =
          \pm\frac{x}{m^2}\in\bR,
          \) for all \(x\in\bR\).\label{Thm-item:h-h-limit}
  \end{enumerate}
\end{Thm}
The proof of Theorem~\ref{Thm:property-h}
will be given in Section~\ref{Subsec:pf-h}.

\begin{Cor}
  Suppose that~\ref{item:assumption}
  is  satisfied.
  For \(-1\le \gamma\le 1\), define
  \begin{align}\label{eq:def-h-gamma}
    h^{(\gamma)}(x)=h(x)+\frac{\gamma}{m^2}x, \quad x\in\bR.
  \end{align}
  Then \(h^{(\gamma)}\) is subadditive and \(h^{(\gamma)}(x)\ge 0\).
\end{Cor}
\begin{proof}
  By definition, we have \(h^{(\gamma)}(0)=0\).
  From (\ref{Thm-item:subadditive-h}) of Theorem~\ref{Thm:exist-h},
  the function \(h^{(\gamma)}\) is subadditive. From
  (\ref{Thm-item:h/x-infinity}) of Theorem~\ref{Thm:property-h},
  it holds that \(\lim_{x\to\pm\infty} h^{(\gamma)}(x)/\abs{x} = (1+\gamma)/m^2 \ge
  0\).
  Suppose \(h^{(\gamma)}(x)<0\) for some \(x\in\bR\setminus\cbra{0}\).
  Since \(h^{(\gamma)}\) is subadditive,
  we have \(h(nx)/n\le h(x)<0\) for \(n=1,2,\dots\).
  Letting \(n\to\infty\), we face the contradiction.
  Therefore we have \(h^{(\gamma)}(x)\ge 0\) for all \(x\in\bR\).
\end{proof}
We will prove in Theorem~\ref{Thm:hg-invariant} that
the function \(h^{(\gamma)}\) is invariant for the process
killed upon hitting zero.
Let \(\cL_{+}^1\) denote the set of non-negative functions
on \([0,\infty)\) which satisfy \(\int_0^\infty f(x)\,\cd x<\infty\).
For \(f\in \cL_{+}^1\), define
\begin{align}\label{eq:def-mart}
  M_t^{(\gamma)}= M_t^{(\gamma, f)}= h^{(\gamma)}(X_t)f(L_t)+
  \int_0^\infty
  f(L_t+u)\,\cd u.
\end{align}
Note that, when \(m^2=\infty\),
we have \(h^{(\gamma)}=h^{(0)}=h\) and
\(M^{(\gamma)}=M^{(0)}\) for all \(\gamma\).
\begin{Thm}\label{Thm:martingale}
  Suppose that~\ref{item:assumption}
  is satisfied.
  Let \(f\in \cL^1_{+}\), \(-1\le \gamma\le 1\) and \(x\in \bR\). Then
  \((M_t^{(\gamma)},t\ge 0)\) is a non-negative \(((\cF_t), \bP_x)\)-martingale.
\end{Thm}
Theorem~\ref{Thm:martingale} will be proved in
Section~\ref{Subsec:pf-hitting-l1}.
Using this martingale,
we discuss our penalization problems.
Let \(L=(L_t)\) denote the local time at the origin of \(X\);
see Section~\ref{Subsec:LT-em}.

\begin{Thm}[hitting time clock]\label{Thm:hitting-time-result}
  Suppose that the condition~\ref{item:assumption} is satisfied.
  Let \(f \in \cL_{+}^1\) and \(x \in \bR\).
  Define \(h^B(a)=\bP\sbra{L_{T_a}}\) and
  \begin{align}
    N_t^a & = h^B(a) \bP_x\sbra*{f(L_{T_a}); t<T_a|\cF_t}, \\
    M_t^a & = h^B(a) \bP_x\sbra*{f(L_{T_a})|\cF_t}.
  \end{align}
  Then it holds that
  \begin{align}
    \lim_{a\to\pm\infty}N_t^a
    =\lim_{a\to\pm\infty} M_t^a
    = M_t^{(\pm 1)},
    \qquad \text{\(\bP_x\)-a.s.\ and in \(\cL^1(\bP_x)\).}
  \end{align}
  Consequently, if \(M_0^{(\pm 1)}>0\) under \(\bP_x\),
  it holds that
  \begin{align}
    \frac{\bP_x\sbra{F_t f(L_{T_a})}}{\bP_x\sbra{f(L_{T_a})}}
    \longrightarrow \bP_x\sbra*{F_t \frac{M_t^{(\pm 1)}}{M_0^{(\pm 1)}}},
    \qquad \text{as \(a \to \pm\infty \),}
  \end{align}
  for all bounded \(\cF_t\)-measurable functionals \(F_t\).
\end{Thm}
Theorem~\ref{Thm:hitting-time-result} will be proved
in Section~\ref{Sec:hitting-time}.
If we take \(f=1_{\cbra{u=0}}\), we obtain
the conditioning result.
\begin{Cor}\label{Cor:hitting-cond}
Suppose that the condition~\ref{item:assumption} is satisfied.
Let \(x\in \bR\) with \(h^{(\pm 1)}(x)>0\).
Then it holds that
\begin{align}
  \bP_x\sbra{F_t|T_0>T_a}
  \longrightarrow \bP_x\sbra*{F_t\frac{h^{(\pm 1)}(X_t)}{h^{(\pm 1)}(x)};T_0>t},
  \qquad\text{as \(a\to\pm\infty\),}
\end{align}
for all bounded \(\cF_t\)-measurable functionals \(F_t\).
\end{Cor}
See also Corollary~\ref{Cor:avoid-zero}.

Let us state our penalization result with two-point hitting time clock.
For \(a,b\in\bR\), we write
\(T_{a,b} = T_{\cbra{a,b}}= T_{a} \wedge T_{b}\).
For \(-1\le\gamma\le 1\), we say
\begin{align}\label{eq:a-b-inf-notation}
  \text{\((a,b)\xrightarrow[]{\gamma}\infty\)
    when \(a\to\infty\), \(b\to\infty\) and \(\frac{a-b}{a+b}\to\gamma\).}
\end{align}
\begin{Thm}[two-point hitting time clock]\label{Thm:two-point-hitting-time-result}
  Suppose that the condition~\ref{item:assumption} is satisfied.
  Let \(f \in \cL_{+}^1\), \(x \in \bR\), and \(a,b>0\).
  Define \(	h^C(a,-b) = \bP\sbra{L_{T_{a, -b}}}\), and
  \begin{align}
    N_t^{a,b}
     & =h^C(a, -b)
    \bP_x\sbra{f(L_{T_{a, -b}}); t<T_{a,-b}|\cF_t}, \\
    M_t^{a,b}
     & =h^C(a, -b)
    \bP_x\sbra{f(L_{T_{a,-b}})|\cF_t}.
  \end{align}
  Then it holds that
  \begin{align}
    \lim_{(a,b)\xrightarrow[]{\gamma}\infty}N_t^{a,b}
    =\lim_{(a,b)\xrightarrow[]{\gamma}\infty} M_t^{a,b}
    =M_t^{(\gamma)},\quad
    \text{\(\bP_x\)-a.s.\ and
      in \(\cL^1(\bP_x)\).}
  \end{align}
  Consequently, if \(M_0^{(\gamma)}>0\) under \(\bP_x\),
  it holds that
  \begin{align}
    \frac{\bP_x\sbra{F_t f(L_{T_{a, -b}})}}{\bP_x\sbra{f(L_{T_{a,-b}})}}
    \longrightarrow \bP_x\sbra*{F_t \frac{M_t^{(\gamma)}}{M_0^{(\gamma)}}},
    \qquad
    \text{as \((a,b)\xrightarrow[]{\gamma}\infty\),}
  \end{align}
  for all bounded \(\cF_t\)-measurable functionals \(F_t\).
\end{Thm}
The proof of Theorem~\ref{Thm:two-point-hitting-time-result}
will be given in Section~\ref{Subsec:pf-two-hitting}.
\begin{Cor}\label{Cor:two-point-hitting-cond}
Suppose that the condition~\ref{item:assumption} is satisfied.
Let \(-1\le\gamma\le 1\) and
\(x\in \bR\) with \(h^{(\gamma)}(x)>0\).
Then it holds that
\begin{align}
  \bP_x\sbra{F_t|T_0>T_{a,-b}}
  \longrightarrow \bP_x\sbra*{F_t\frac{h^{(\gamma)}(X_t)}{h^{(\gamma)}(x)};T_0>t},
  \qquad\text{as \((a,b)\xrightarrow[]{\gamma}\infty\),}
\end{align}
for all bounded \(\cF_t\)-measurable functionals \(F_t\).
\end{Cor}
See also Corollary~\ref{Cor:avoid-zero}.

Note that Theorems~\ref{Thm:hitting-time-result}
and~\ref{Thm:two-point-hitting-time-result} show that the limit
law varies according to the chosen clock
when \(m^2<\infty\).

\subsection{Backgrounds of the renormalized zero resolvent}\label{Subsec:back-h}
The existence of \(h\) for symmetric L\'{e}vy processes
was proved by Salminen--Yor~\cite{MR2409011}
under the assumption~\ref{item:assumption}; see also
Yano~\cite{MR2603019}.
We shall review early studies of the existence of \(h\) and
its limit properties
for asymmetric processes.

Similar results were obtained for random walks
by Spitzer~\cite[Chapter VII]{MR0171290},
Port--Stone~\cite{MR0215375,MR0226740,MR261706}
and Stone~\cite{MR238398}.
For L\'{e}vy processes, Port--Stone~\cite[Section 17]{MR346919} obtained some results
which were similar to but different from
Theorems~\ref{Thm:exist-h} and~\ref{Thm:property-h},
reducing them to the random walk case.
(For the proofs
of Theorems~\ref{Thm:exist-h} and~\ref{Thm:property-h},
we are inspired
by Spitzer~\cite[Chapter VII]{MR0171290},
Port--Stone~\cite[Section 17]{MR0215375,MR0226740,MR261706,MR346919}
and Stone~\cite{MR238398}.)

Yano~\cite{MR3072331} showed
the existence of the renormalized zero resolvent \(h\)
under the following two conditions:
\begin{enumerate}[label=\textbf{(Y\arabic*)}]
  \item \(\displaystyle \int_0^\infty \frac{1}{q+\theta(\lambda)}
        \, \cd \lambda < \infty\) for all \(q>0\);\label{item:cond-Yano1}
  \item \(\theta\) and \(\omega\) have measurable derivatives
        on \((0, \infty)\) which satisfy
        \begin{align}
          \int_0^\infty \frac{\rbra{
            \abs{\theta^\prime(\lambda)}+
            \abs{\omega^\prime(\lambda)}}(\lambda^2\wedge 1)
          }{{\theta(\lambda)}^2+{\omega(\lambda)}^2}
          \, \cd \lambda < \infty.
        \end{align}\label{item:cond-Yano2}
\end{enumerate}
Pant\'{\i}~\cite{MR3689384} proved the existence of \(h\)
under the condition
\begin{enumerate}[label=\textbf{(P)}]
  \item \ref{item:cond-Yano1} and \( \displaystyle
        \int_\bR \abs*{\Re\rbra*{\frac{1-\ce^{\ci \lambda}}{\varPsi(\lambda)}}}
        \, \cd \lambda < \infty\),\label{item:Panti-cond}
\end{enumerate}
which is weaker than~\ref{item:cond-Yano1} and~\ref{item:cond-Yano2},
and applied it to the conditioning to avoid zero
with exponential clock.
Tsukada~\cite{MR3838874} also proved the existence
of \(h\)
under the assumption
\begin{enumerate}[label=\textbf{(T)}]
  \item \ref{item:assumption} and
        \(\displaystyle
        \int_0^1 \abs*{\Im\rbra*{\frac{\lambda}{\varPsi(\lambda)}}}
        \, \cd \lambda <
        \infty,
        \)\label{item:Tsukada-b}
\end{enumerate}
which is weaker than~\ref{item:Panti-cond};
see~\cite[Proposition 15.3]{MR3838874}.

\begin{Rem}
  We do not know whether the integral
  representation~\eqref{eq:h-repre} also holds
  in the case \(m^2=\infty\).
  Yano~\cite{MR3072331} showed that, if \(X\) is symmetric,
  then~\eqref{eq:h-repre} holds.
  Tsukada~\cite{MR3838874} showed that~\ref{item:Tsukada-b} implies
  \eqref{eq:h-repre}.
\end{Rem}

\subsection{Organization}
The remainder of this paper is organized as follows.
In Section~\ref{Sec:preliminaries}, we prepare
certain general properties and
preliminary facts of L\'{e}vy processes.
In Section~\ref{Sec:h}, we
study the renormalized zero resolvent.
In Sections~\ref{Sec:exp-clock},~\ref{Sec:hitting-time},~\ref{Sec:two-hitting-time}
and~\ref{Sec:inv-time}, we discuss the penalization results
with exponential clock, hitting time clock, two-point hitting time
clock and inverse local time clock, respectively.
In Section~\ref{Sec:universal-measure}, we introduce certain universal
\(\sigma\)-finite measures to study long time
behaviors of sample paths of the penalized measure.
In Section~\ref{Sec:trans}, we study penalization in the transient case.
In Section~\ref{Sec:mart} as an appendix, we study martingale property
of \((X_t f(L_t),t\ge 0)\).

\subsection*{Acknowledgments}
The authors would like to thank Hiroshi Tsukada
for his helpful comments.

\section{Preliminaries}\label{Sec:preliminaries}
\subsection{Absolutely continuous resolvent}\label{Subsec:resolvent}
We now consider the following two conditions:
\begin{enumerate}[label=\textbf{(A\arabic*)},series=cond]
  \item\label{item:cond-not-CPP}
        The process \(X\) is not a compound Poisson process;
  \item\label{item:cond-regular}
        \(0\) is regular for itself, i.e.,
        \(\bP(T_0 = 0) = 1\).
\end{enumerate}

The next lemma is due to Kesten~\cite{MR0272059} and
Bretagnolle~\cite{MR0368175}.
\begin{Lem}[\cite{MR0272059,MR0368175}]\label{Lem:iff-cond-notCPP-regular}
  The conditions~\ref{item:cond-not-CPP} and~\ref{item:cond-regular}
  hold if and only if the following two assertions hold:
  \begin{enumerate}[resume*=cond]
    \item For each \(q>0\),
          the characteristic exponent \(\varPsi\) satisfies
          \begin{align}
            \int_\bR \Re\rbra*{\frac{1}{q+\varPsi(\lambda)}}\, \cd \lambda < \infty;
          \end{align}\label{item:cond-Reinte}
    \item We have either \(\sigma>0\) or
          \(\int_{(-1, 1)} \abs{x} \nu(\cd x) = \infty\).\label{item:cond-nu}
  \end{enumerate}
  Furthermore, under the condition~\ref{item:cond-Reinte},
  the condition~\ref{item:cond-regular} holds if and only
  if the condition~\ref{item:cond-nu} holds.
\end{Lem}
If the above conditions hold, it is known that
\(X\) has a bounded continuous resolvent density. See,
e.g., Theorem II.16 and Theorem II.19 of Bertoin~\cite{MR1406564}
\begin{Lem}[\cite{MR1406564}]\label{Lem:resolvent-density}
  The condition~\ref{item:cond-Reinte} holds if and only if
  \(X\) has the bounded \(q\)-resolvent density \(r_q\), for \( q > 0\),
  which satisfies
  \begin{align}
    \int_\bR f(x) r_q(x) \, \cd x
    = \bP\sbra*{\int_0^\infty \ce^{-qt}f(X_t) \, \cd t}
  \end{align}
  for all non-negative measurable functions \(f\).
  Moreover, under the condition~\ref{item:cond-Reinte},
  the condition~\ref{item:cond-regular} holds if and only if
  \(x \mapsto r_q(x)\) is continuous.
\end{Lem}

If \(r_q(x)\) is bounded in \(x\in\bR\),~\cite[Corollary II.18]{MR1406564}
implies that
the Laplace transform of \(T_0\)
can be represented as
\begin{align}\label{eq:exp-hit-time}
  \bP_x\sbra{\ce^{-qT_0}} = \frac{r_q(-x)}{r_q(0)}, \quad q>0, \; x \in \bR.
\end{align}

\begin{Prop}\label{Prop:r_q-represent}
  Suppose that the condition~\ref{item:assumption} holds.
  Then the bounded continuous resolvent density can be expressed as
  \begin{align}
    r_q(x) =
    \frac{1}{2\pi}\int_{-\infty}^\infty
    \frac{\ce^{-\ci\lambda x}}{q+\varPsi(\lambda)}
    \, \cd \lambda
    =\frac{1}{\pi}\int_0^\infty
    \Re\rbra*{\frac{\ce^{-\ci\lambda x}}{q+\varPsi(\lambda)}}
    \, \cd \lambda
  \end{align}
  for all \(q > 0\) and \(x\in \bR\).
\end{Prop}
Proposition~\ref{Prop:r_q-represent} can be proved using Fourier
inversion formula; see, e.g.,~\cite[Lemma 2]{MR1894112}
and~\cite[Corollary 15.1]{MR3838874}.
By Lemmas~\ref{Lem:iff-cond-notCPP-regular}
and~\ref{Lem:resolvent-density} and Proposition~\ref{Prop:r_q-represent},
the condition
~\ref{item:assumption}
implies~\ref{item:cond-not-CPP}--\ref{item:cond-nu}.
\begin{Lem}[Tsukada {\cite[Lemma 15.5]{MR3838874}}]\label{Lem:inte-of-varPsi}
  Suppose that the condition~\ref{item:assumption} holds.
  Then the following assertions hold:
  \begin{enumerate}
    \item \(\abs{\varPsi(\lambda)} \to \infty\)
          as \(\lambda\to\pm\infty\);\label{item:varpsi-to-infty}
    \item \(\displaystyle \int_\delta^\infty \abs*{\frac{1}{\varPsi(\lambda)}}
          \, \cd \lambda < \infty\)
          for all \(\delta>0\);\label{Lem-item:infty-inte}
    \item \(\displaystyle \int_0^\delta \abs*{\frac{\lambda^2}{\varPsi(\lambda)}}
          \, \cd \lambda < \infty\)
          for all \(\delta>0\);\label{Lem-item:zero-inte}
    \item \(\displaystyle \lim_{q\to 0+}
          \int_0^\infty \abs*{\frac{q}{q+\varPsi(\lambda)}}
          \, \cd \lambda
          = 0\).
          In particular, \(qr_q(x) \to 0\) as \(q \to 0+\).\label{Lem-item:qr_q}
  \end{enumerate}
\end{Lem}

\subsection{Local time and its excursion measure}\label{Subsec:LT-em}
Assume the conditions~\ref{item:cond-not-CPP} and~\ref{item:cond-regular} hold.
Then we can define local time at \(0\), which we denote by \(L=(L_t,t\ge 0)\).
Note that \(L\) is continuous in \(t\) and satisfies
\begin{align}\label{eq:regularity-of-L}
  \bP_x\sbra*{\int_0^\infty \ce^{-qt} \cd L_t} = r_q(-x),
  \quad q>0,\;x \in \bR.
\end{align}
See, e.g.,~\cite[Section V]{MR1406564}.
In particular, \(r_q(x)\) is non-decreasing as \(q \to 0+\).
Let \(\eta=\rbra{\eta_l, l \ge 0}\) denote
the right-continuous inverse of \(L\)
which is given as
\(\eta_l = \inf\cbra{t > 0\colon L_t > l}\).
Then the process \((\eta, \bP)\) is a possibly killed
subordinator, and its Laplace transform is
\(\bP\sbra{\ce^{-q\eta_l}} = \ce^{-l/r_q(0)}\), for
\(l,q > 0\),
see, e.g.,~\cite[Proposition V.4]{MR1406564}.

Now we can apply It\^{o}'s excursion theory.
Let \(n\) denote the characteristic measure of excursions away from the origin.
We denote \(e_l\) for excursion which starts at local time \(l\).
Then we see that the subordinator \(\eta\) has no drift and its L\'{e}vy measure is
\(n(T_0\in \cd x)\).
In particular, we have
\begin{align}\label{eq:exp-nt0}
  \ce^{-l/r_q(0)} = \bP\sbra{\ce^{-q\eta_l}}
  = \exp\rbra{-l n\sbra{1-\ce^{-qT_0}}},
  \quad l \ge 0.
\end{align}
This implies that
\begin{align}\label{eq:rel-n-r_q}
  n\sbra{1 - \ce^{-qT_0}} = \frac{1}{r_q(0)},
\end{align}
which is also obtained from~\cite[(3.16)]{MR2552915}.
Now set
\begin{align}
  \kappa = \lim_{q\to 0+}\frac{1}{r_q(0)} = n(T_0 =\infty).\label{eq:kappa}
\end{align}
It is known that \(\kappa = 0\) (resp.\ \(\kappa > 0\))
if and only if \(X\) is recurrent (resp.\ transient);
see, e.g.,~\cite[Theorem I.17]{MR1406564}
and~\cite[Theorem 37.5]{MR1739520}.
It is also known that
\(X\) is recurrent if and only if \(X\) is point recurrent;
see, e.g.,~\cite[Remark 43.12]{MR1739520}
(see also~\cite[Excercise II.6.4]{MR1406564}).
Under the assumption~\ref{item:assumption},
we can prove this fact by using Theorem~\ref{Thm:exist-h};
in fact,
by equations~\eqref{eq:exp-hit-time} and~\eqref{eq:kappa},
it holds that, for \(x\in\bR\),
\begin{align}
  \bP_x(T_0<\infty) &= \lim_{q\to 0+}\bP_x\sbra{\ce^{-qT_0}}
  = \lim_{q\to 0+}\frac{r_q(-x)}{r_q(0)}
  = 1+\lim_{q\to 0+}\frac{h_q(x)}{r_q(0)}
  =1.
\end{align}

We define \(D = \cbra{l\ge 0\colon \eta_{l-} < \eta_l}\).
Then the following formula is well-known in the excursion theory.
\begin{Lem}[{Compensation formula; see e.g.,
  Bertoin~\cite[Corollary IV.11]{MR1406564}}]\label{Lem:compensation-formula}
  Let \(F(t,\omega,e)\) be a measurable functional
  on \([0, \infty) \times \cD \times \cD\) such that,
  for every fixed \(e\in \cD\), the process
  \((F(t,\cdot,e), t\ge 0)\) is \((\cF_t)\)-predictable. Then
  \begin{align}
    \bP\sbra*{\sum_{l\in D} F(\eta_{l-}, X, e_l)}
    = \bP\otimes \widetilde{n}
    \sbra*{\int_0^\infty \cd L_t\, F(t, X, \widetilde{X})},
  \end{align}
  where the symbol \(\widetilde{\hspace{12pt}}\) means independence.
\end{Lem}

Let \(L^a =\rbra{L_t^a, t\ge 0}\) denote the local time at \(a\in \bR\)
which is normalized by
\begin{align}\label{eq:regularity-of-L^a}
  \bP_x\sbra*{\int_0^\infty \ce^{-qt} \cd L^a_t} = r_q(a-x),
  \quad  q>0,\;x \in \bR.
\end{align}
We denote by \(\eta^a=(\eta^a_u, u\ge 0)\)
the right-continuous inverse of \(L^a\) given by
\(\eta^a_u=\inf\cbra{t>0\colon L_t^a>u}\).
We denote by \(n^a\) the characteristic measure of excursions away from \(a\).

\section{The renormalized zero resolvent}\label{Sec:h}
Let us consider the existence and properties of
the renormalized zero resolvent in Theorems~\ref{Thm:exist-h} and~\ref{Thm:property-h}.
Recall that we assume \(X\) is recurrent, i.e., \(\kappa=0\),
and assume the condition~\ref{item:assumption}.
\subsection{Key lemmas for the renormalized zero resolvent}
To show Theorems~\ref{Thm:exist-h} and~\ref{Thm:property-h},
we prepare some lemmas.
Recall that \(m^2\) has been introduced in~\eqref{eq:second-moment};
\(m^2=\bP\sbra{X_1^2}\).
\begin{Lem}\label{Lem:psi-conv}
  The following assertions hold.
  \begin{enumerate}
    \item If \(m^2<\infty\), then
          \begin{align}
            \varPsi(\lambda)
            = \frac{1}{2}\sigma^2\lambda^2
            + \int_\bR \rbra{1-\ce^{\ci \lambda x}+\ci \lambda x}\nu(\cd x),
          \end{align}\label{Lem-item:psi-repre}
          and
          \begin{align}\label{eq:psi/lambda2}
            \lim_{\lambda \to 0}\frac{\varPsi(\lambda)}{\lambda^2}
            =\lim_{\lambda \to 0}\frac{\theta(\lambda)}{\lambda^2}
            = \frac{m^2}{2}
            =\frac{1}{2}\rbra*{\sigma^2 + \int_\bR x^2 \nu(\cd x)}.
          \end{align}\label{Lem-item:psi-conv}
    \item If \(m^2=\infty\), then
          \begin{align}
            \lim_{\lambda \to 0}\frac{\lambda^2}{\varPsi(\lambda)}
            =\lim_{\lambda \to 0}\frac{\lambda^2}{\theta(\lambda)}=0.
          \end{align}
  \end{enumerate}
\end{Lem}

\begin{proof}
  It is well-known
  (see, e.g.,~\cite[Theorem 2.7]{MR3155252})
  that \(X\) has finite variance
  if and only if \(\int_\bR x^2 \nu(\cd x) < \infty\).
  We first assume that
  \(\int_\bR x^2 \nu(\cd x) < \infty\). Then we know
  \(\bP\sbra{X_1}\) and \(\bP\sbra{X_1^2}\) are finite and
  \begin{gather}
    \bP\sbra{X_1}
    =\ci\varPsi^\prime(0)
    = - v + \int_{\bR\setminus (-1, 1)} x \nu(\cd x), \\
    \bP\sbra{X_1^2}
    = \varPsi^{\prime\prime}(0)
    = \sigma^2 + \int_\bR x^2 \nu(\cd x) +
    {\bP\sbra{X_1}}^2.
  \end{gather}
  Since \(X\) is recurrent, we have \(\bP[X_1] = 0\);
  see, e.g.,~\cite[Problem 7.2]{MR3155252}.
  This implies that
  \begin{align}
    \varPsi(\lambda)
    = \frac{1}{2} \sigma^2 \lambda
    + \int_\bR \rbra{1-\ce^{\ci\lambda x} + \ci \lambda x} \nu(\cd x).
  \end{align}
  By l'Hôpital's rule, we obtain
  \begin{align}
    \lim_{\lambda\to 0} \frac{\varPsi(\lambda)}{\lambda^2}
    = \lim_{\lambda\to 0} \frac{\varPsi^\prime(\lambda)}{2\lambda}
    = \frac{\varPsi^{\prime\prime}(0)}{2} = \frac{m^2}{2}.
  \end{align}
  Taking real parts on both sides, we also have
  \begin{align}
    \lim_{\lambda\to 0}\frac{\theta(\lambda)}{\lambda^2} = \frac{m^2}{2}.
  \end{align}
  We next assume that \(\int_\bR x^2 \nu(\cd x) = \infty\).
  Then we know \(m^2 = \infty\).
  By~\eqref{eq:theta}, we have
  \begin{align}
    \abs*{\frac{\varPsi(\lambda)}{\lambda^2}}
    \ge \abs*{\frac{\theta(\lambda)}{\lambda^2}}
    \ge \int_\bR \frac{1-\cos \lambda x}{\lambda^2}\nu(\cd x).
  \end{align}
  Using Fatou's lemma and l'Hôpital's rule, we obtain
  \begin{align}
    \liminf_{\lambda\to 0}\abs*{\frac{\varPsi(\lambda)}{\lambda^2}}
    \ge \liminf_{\lambda\to 0}\abs*{\frac{\theta(\lambda)}{\lambda^2}}
     & \ge \int_\bR \liminf_{\lambda\to 0}
    \frac{1-\cos \lambda x}{\lambda^2}
    \nu(\cd x)                                       \\
     & = \int_\bR \frac{x^2}{2} \nu(\cd x) = \infty.
  \end{align}
  Therefore the proof is complete.
\end{proof}

When \(m^2<\infty\), the next lemma
is essential for the renormalized zero resolvent.
\begin{Lem}\label{Lem:fin-omega-integral}
  Assume \(m^2 < \infty\).
  Then it holds that
  \begin{align}\label{eq:fin-omega-integral}
    \int_\bR \abs*{\frac{\omega(\lambda)}{\lambda^3}} \, \cd \lambda \
    < \infty.
  \end{align}
  Consequently,
  Tsukada's condition~\ref{item:Tsukada-b} holds.
\end{Lem}

\begin{proof}[Proof of Lemma~\ref{Lem:fin-omega-integral}]
  By Lemma~\ref{Lem:psi-conv},
  we have
  \begin{align}
    \omega(\lambda)
    = \int_{\bR} \rbra{\lambda x - \sin \lambda x}
    \nu(\cd x).
  \end{align}
  Hence we have
  \begin{align}
    \int_\bR \abs*{\frac{\omega(\lambda)}{\lambda^3}} \, \cd \lambda
     & = \int_\bR \abs*{\int_\bR
      \frac{\lambda x - \sin \lambda x}{\lambda^3}
      \nu(\cd x)}
    \, \cd \lambda                                  \\
     & \le \rbra*{\int_{-\infty}^0 + \int_0^\infty}
    \rbra*{\int_{-\infty}^0 + \int_0^\infty}
    \abs*{\frac{\lambda x - \sin \lambda x}{\lambda^3}}
    \nu(\cd x)
    \, \cd \lambda.
  \end{align}
  Since \(\lambda x - \sin \lambda x \ge 0\)
  for \((x, \lambda) \in {(0, \infty)}^2\),
  it holds that
  \begin{align}
    \int_0^\infty \int_0^\infty
    \abs*{\frac{\lambda x - \sin \lambda x}{\lambda^3}}
    \nu(\cd x)
    \, \cd \lambda
     & =
    \int_0^\infty \int_0^\infty
    \frac{\lambda x - \sin \lambda x}{\lambda^3}
    \, \cd \lambda \,
    \nu(\cd x)                                                \\
     & =
    \int_0^\infty x^2 \nu(\cd x)
    \int_0^\infty
    \frac{\xi - \sin \xi}{\xi^3} \, \cd\xi                    \\
     & = \frac{4}{\pi} \int_0^\infty x^2 \nu(\cd x) < \infty.
  \end{align}
  Other integrals are also proved to be finite by the same discussions,
  and we obtain~\eqref{eq:fin-omega-integral}.
  By~(\ref{Lem-item:psi-repre}) of Lemma~\ref{Lem:psi-conv},
  we see that \(\lambda^2/\varPsi(\lambda)\) is bounded near \(\lambda=0\).
  Thus we have
  \begin{align}
    \int_0^1 \abs*{\Im\rbra*{\frac{\lambda}{\varPsi(\lambda)}}}\,\cd\lambda
    = \int_0^1 \abs*{\frac{\lambda^2}{\varPsi(\lambda)}}^2
    \abs*{\frac{\omega(\lambda)}{\lambda^3}}\, \cd \lambda <\infty.
  \end{align}
  Since~\ref{item:assumption} is assumed in this section,
  this implies that Tsukada's condition~\ref{item:Tsukada-b} holds.
\end{proof}

\begin{Lem}\label{Lem:exist-hheq}
  The following assertions hold.
  \begin{enumerate}
    \item \(\displaystyle
          h^S(x)
          \coloneqq \lim_{q\to 0+}(h_q(x) + h_q(-x))
          = \frac{2}{\pi}
          \int_0^\infty
          \Re\rbra*{\frac{1-\cos \lambda x}{\varPsi(\lambda)}} \, \cd \lambda,
          \quad\text{for \(x\in\bR\).}
          \)\label{Lem-item:exist-h^S}
    \item \(\displaystyle
          \lim_{x\to\pm\infty} \frac{h^S(x)}{\abs{x}} =
          \frac{2}{m^2} \in [0,\infty).
          \)\label{Lem-item:h^S/x}
    \item For \(x, y \in \bR\),
          \begin{align}
            h^D(x, y)
             & \coloneqq \lim_{q\to 0+}\cbra{h_q(y+2x) - 2h_q(y+x) + h_q(y)} \\
             & = \frac{2}{\pi} \int_0^\infty
            \Re\rbra*{\ce^{\ci \lambda(y+x)}
              \frac{1-\cos \lambda x}{\varPsi(\lambda)}} \, \cd \lambda.
          \end{align}
          Moreover, it holds that
          \(\lim_{y\to\pm\infty} h^D(x, y) = 0\).\label{Lem-item:h^D}
  \end{enumerate}
\end{Lem}

\begin{proof}
  \noindent (\ref{Lem-item:exist-h^S})
  By~\eqref{eq:def-h_q}, it holds that
  \begin{align}\label{eq:h_q^s}
    h_q(x)+h_q(-x)
    = \frac{2}{\pi}\int_0^\infty
    \Re\rbra*{\frac{1-\cos \lambda x}{q+\varPsi(\lambda)}}
    \,\cd\lambda.
  \end{align}
  Since \(\theta(\lambda) \ge 0\), we have
  \(\abs{q+\varPsi(\lambda)}\ge \abs{\varPsi(\lambda)}\).
  Hence it holds that
  \begin{align}\label{eq:hs-conv-DCT}
    \abs*{\Re\rbra*{\frac{1-\cos \lambda x}{q+\varPsi(\lambda)}}}
    \le \abs*{\frac{1-\cos \lambda x}{q+\varPsi(\lambda)}}
    \le \frac{1-\cos \lambda x}{\abs{\varPsi(\lambda)}}
    \le \frac{{(\lambda x)}^2 \wedge 2}{\abs{\varPsi(\lambda)}},
  \end{align}
  which is integrable in \(\lambda>0\) by Lemma~\ref{Lem:inte-of-varPsi}.
  Then we may apply the dominated convergence theorem
  to deduce that
  \begin{align}
    h_q(x)+h_q(-x) \longrightarrow
    \frac{2}{\pi}\int_0^\infty
    \Re\rbra*{\frac{1-\cos \lambda x}{\varPsi(\lambda)}}\,\cd\lambda,
    \quad \text{as \(q \to 0+\).}
  \end{align}

  \noindent  (\ref{Lem-item:h^S/x})
  We only consider the case \(x \to \infty\) since the case
  \(x \to -\infty\) can be proved in the same way.
  For any \(\delta>0\), we have
  \begin{align}
    \abs*{\frac{2}{\pi x}
    \int_\delta^\infty
    \Re\rbra*{\frac{1-\cos \lambda x}{\varPsi(\lambda)}}
    \, \cd \lambda}
    \le \frac{2}{\pi\abs{x}}
    \int_\delta^\infty \abs*{\frac{2}{\varPsi(\lambda)}}
    \, \cd \lambda
    \longrightarrow 0, \quad \text{as \(x\to\infty\).}\label{eq:h^S/x-big}
  \end{align}
  Fix \(\varepsilon > 0\).
  By Lemma~\ref{Lem:psi-conv},
  we can choose \(\delta>0\) such that
  \begin{align}
    \abs*{\frac{\lambda^2}{\varPsi(\lambda)}
      - \frac{2}{m^2}} < \varepsilon,
    \quad \text{for } \abs{\lambda} < \delta.
  \end{align}
  Then it holds that
  \begin{align}
    \frac{2}{\pi x}
    \int_0^\delta
    \Re\rbra*{\frac{1-\cos \lambda x}{\varPsi(\lambda)}}
    \, \cd \lambda
     & \le \rbra*{\frac{2}{m^2} + \varepsilon}
    \frac{2}{\pi x}\int_0^\delta
    \frac{1-\cos \lambda x}{\lambda^2}
    \, \cd \lambda                             \\
     & \to \rbra*{\frac{2}{m^2} + \varepsilon}
    \frac{2}{\pi}
    \int_0^\infty \frac{1-\cos \xi}{\xi^2}
    \, \cd \xi                                 \\
     & = \frac{2}{m^2} + \varepsilon,
    \quad \text{as \(x\to\infty\).}
  \end{align}
  Thus we obtain
  \begin{align}
    \limsup_{x\to\infty}
    \frac{2}{\pi x}
    \int_0^\delta
    \Re\rbra*{\frac{1-\cos \lambda x}{\varPsi(\lambda)}}
    \, \cd \lambda
    \le
    \frac{2}{m^2} + \varepsilon.\label{eq:h^S/x-upper}
  \end{align}
  In the same way, we can show that
  \begin{align}
    \liminf_{x\to\infty}
    \frac{2}{\pi x}
    \int_0^\delta
    \Re\rbra*{\frac{1-\cos \lambda x}{\varPsi(\lambda)}}
    \, \cd \lambda
    \ge
    \frac{2}{m^2} - \varepsilon.\label{eq:h^S/x-lower}
  \end{align}
  By~\eqref{eq:h^S/x-big},~\eqref{eq:h^S/x-upper}
  and~\eqref{eq:h^S/x-lower},
  the result follows.

  \noindent (\ref{Lem-item:h^D})
  By~\eqref{eq:def-h_q}, we have
  \begin{align}
    h_q(y+2x) - 2h_q(y+x) + h_q(y)
    =  \frac{2}{\pi}
    \int_0^\infty
    \Re\rbra*{ \ce^{\ci\lambda (y+x)}
      \frac{1-\cos \lambda x}{q+\varPsi(\lambda)}}
    \, \cd \lambda.
  \end{align}
  By the same way as~\eqref{eq:hs-conv-DCT},
  we may apply the dominated convergence theorem
  to obtain
  \begin{align}
    h^D(x, y)
     & = \lim_{q\to 0+} \cbra{h_q(y+2x) - 2h_q(y+x) + h_q(y)} \\
     & = \frac{2}{\pi} \int_0^\infty
    \Re\rbra*{\ce^{\ci \lambda(y+x)}
      \frac{1-\cos \lambda x}{\varPsi(\lambda)}} \, \cd \lambda.
  \end{align}
  Furthermore, by the Riemann--Lebesgue lemma, we obtain
  \(h^D(x, y) \longrightarrow 0\)
  as \(y \to \pm \infty\).
\end{proof}

\subsection{Proofs of Theorems~\ref{Thm:exist-h}
  and~\ref{Thm:property-h}}\label{Subsec:pf-h}
We separate the proof into the two cases:
\(m^2<\infty\) and \(m^2=\infty\).
We first show the existence and properties of \(h\) in the case \(m^2<\infty\).
In this case, we can use the dominated convergence theorem.
\begin{proof}[Proof of~(\ref{Thm-item:exist-h}) of Theorem~\ref{Thm:exist-h}
    in the case \(m^2<\infty\)]
  For each \(x, \lambda \in \bR\),
  we observe that
  \begin{align}
    \Re\rbra*{\frac{1-\ce^{\ci \lambda x}}{q+\varPsi(\lambda)}}
    = \frac{(q+\theta(\lambda)) (1-\cos\lambda x)
      + \omega(\lambda) \sin\lambda x}
    {\abs{q+\varPsi(\lambda)}^2}.
  \end{align}
  Hence it follows from \(\theta(\lambda) \ge 0\) that
  \begin{align}
     & \abs*{\Re\rbra*{
    \frac{1-\ce^{\ci \lambda x}}{q+\varPsi(\lambda)}}}    \\
     & \le
    \rbra*{\frac{1-\cos\lambda x}{\abs{\varPsi(\lambda)}}
      + \frac{\lambda^4}
      {\abs{\varPsi(\lambda)}^2}
      \abs*{\frac{\omega(\lambda)}{\lambda^3}}
      \abs*{\frac{\sin\lambda x}{\lambda}}}
    \wedge
    \abs*{\frac{1-\ce^{\ci \lambda x}}{\varPsi(\lambda)}} \\
     & \le
    \rbra*{\abs*{\frac{\lambda^2 x^2}{\varPsi(\lambda)}}
      + \abs*{\frac{\lambda^2}{\varPsi(\lambda)}}^2
      \abs*{\frac{\omega(\lambda)}{\lambda^3}}
      \abs{x}}
    \wedge
    \abs*{\frac{2}{\varPsi(\lambda)}}.
  \end{align}
  By Lemma~\ref{Lem:inte-of-varPsi},
  (\ref{Lem-item:psi-conv}) of Lemma~\ref{Lem:psi-conv}
  and Lemma~\ref{Lem:fin-omega-integral},
  the last quantity is integrable in \(\lambda>0\).
  Therefore, we may apply the dominated convergence theorem
  to conclude that
  \begin{align}
    h_q(x)
     & =
    \frac{1}{\pi}
    \int_0^\infty
    \Re\rbra*{\frac{1-\ce^{\ci \lambda x}}{q+\varPsi(\lambda)}}
    \,\cd \lambda \\
     & \to
    \frac{1}{\pi}
    \int_0^\infty
    \Re\rbra*{\frac{1-\ce^{\ci \lambda x}}{\varPsi(\lambda)}}
    \,\cd \lambda,
    \quad \text{as \(q\to 0+\).}
  \end{align}
  Hence the proof is complete.
\end{proof}

\begin{proof}[Proof of Theorem~\ref{Thm:property-h}
    in the case \(m^2<\infty\)]
  \noindent (\ref{Thm-item:h/x-infinity})
  We take \(\delta>0\) sufficiently small.
  By (\ref{Thm-item:exist-h}) of Theorem~\ref{Thm:exist-h},
  we have
  \begin{align}
    \hspace{-20pt}\frac{h(x)}{x}
     & =
    \frac{1}{\pi x}
    \int_0^\infty
    \Re\rbra*{\frac{1-\ce^{\ci\lambda x}}{\varPsi(\lambda)}}
    \, \cd \lambda \\
     & =
    \frac{1}{\pi x}\cbra*{
      \int_\delta^\infty
      \Re\rbra*{\frac{1-\ce^{\ci\lambda x}}{\varPsi(\lambda)}}
      \, \cd \lambda
      +
      \int_0^\delta
      \frac{\omega(\lambda)\sin\lambda x}
      {\abs{\varPsi(\lambda)}^2}
      \, \cd \lambda
      +
      \int_0^\delta
      \frac{\theta(\lambda)(1-\cos\lambda x)}
      {\abs{\varPsi(\lambda)}^2}
      \, \cd \lambda}.\label{eq:three-inte}
  \end{align}
  For the first integral in~\eqref{eq:three-inte}, we have
  \begin{align}
    \abs*{
      \frac{1}{\pi x}
      \int_\delta^\infty
      \Re\rbra*{\frac{1-\ce^{\ci\lambda x}}{\varPsi(\lambda)}}
      \, \cd \lambda
    }
    \le
    \frac{1}{\pi\abs{x}}
    \int_\delta^\infty
    \abs*{\frac{2}{\varPsi(\lambda)}}
    \, \cd \lambda
    \longrightarrow 0,
    \quad \text{as \(x \to \pm\infty\).}
  \end{align}
  For the second integral in~\eqref{eq:three-inte}, since we have
  \begin{align}\label{eq:sin-omega-DCT}
    \frac{\abs{\omega(\lambda)\sin\lambda x}}
    {\abs{\varPsi(\lambda)}^2\abs{x}}
    \le
    \abs*{\frac{\lambda^2}{\varPsi(\lambda)}}^2
    \abs*{\frac{\omega(\lambda)}{\lambda^3}},
  \end{align}
  which is integrable in \(\lambda\in (0, \delta)\)
  by (\ref{Lem-item:psi-conv}) of Lemma~\ref{Lem:psi-conv} and
  Lemma~\ref{Lem:fin-omega-integral},
  we can apply the dominated convergence theorem to obtain
  \begin{align}
    \frac{1}{\pi x}
    \int_0^\delta
    \frac{\omega(\lambda)\sin\lambda x}
    {\abs{\varPsi(\lambda)}^2}
    \, \cd \lambda
    \longrightarrow 0,
    \quad \text{as \( x \to \pm\infty\).}
  \end{align}
  For the third integral in~\eqref{eq:three-inte},
  we can apply the similar discussion as
  the proof of (\ref{Lem-item:h^S/x}) of Lemma~\ref{Lem:exist-hheq}.
  Therefore we obtain
  \begin{align}
    \lim_{x\to\pm\infty} \frac{h(x)}{\abs{x}} = \frac{1}{m^2}.
  \end{align}

  \noindent  (\ref{Thm-item:h-h-limit})
  Take \(\delta>0\) sufficiently small.
  Then we have
  \begin{align}
     & h(y+x) - h(y) \\
     & =
    \frac{1}{\pi}
    \int_0^\infty
    \Re\rbra*{\ce^{\ci\lambda y}
      \frac{1-\ce^{\ci\lambda x}}{\varPsi(\lambda)}}
    \, \cd \lambda   \\
     & =
    \frac{1}{\pi}
    \int_0^\infty
    \Re\rbra*{\ce^{\ci\lambda y}
      \frac{1-\cos\lambda x}{\varPsi(\lambda)}}
    \, \cd \lambda
    + \frac{1}{\pi}
    \rbra*{\int_0^\delta + \int_\delta^\infty}
    \Im\rbra*{\ce^{\ci\lambda y}\frac{\sin\lambda x}{\varPsi(\lambda)}}
    \, \cd \lambda.
  \end{align}
  By the Riemann--Lebesgue lemma, we obtain
  \begin{align}
    \lim_{y\to\pm\infty}\int_0^\infty
    \Re\rbra*{\ce^{\ci\lambda y}
      \frac{1-\cos\lambda x}{\varPsi(\lambda)}}
    \, \cd \lambda
    =\lim_{y\to\pm\infty}\int_\delta^\infty
    \Im\rbra*{\ce^{\ci\lambda y}\frac{\sin\lambda x}{\varPsi(\lambda)}}
    \, \cd \lambda
    =0.
  \end{align}
  On the other hand, we have
  \begin{align}
   \hspace{-25pt}\frac{1}{\pi}
    \int_0^\delta
    \Im\rbra*{\ce^{\ci\lambda y}\frac{\sin\lambda x}{\varPsi(\lambda)}}
    \, \cd \lambda
     & =
    \frac{1}{\pi}
    \int_0^\delta
    \frac{\theta(\lambda)\sin\lambda x \sin\lambda y}
    {\abs{\varPsi(\lambda)}^2}
    \, \cd \lambda
    - \frac{1}{\pi}
    \int_0^\delta
    \frac{\omega(\lambda)\sin\lambda x \cos\lambda y}
    {\abs{\varPsi(\lambda)}^2}
    \, \cd \lambda.\label{eq:h-property-finvar}
  \end{align}
  By~\eqref{eq:sin-omega-DCT},
  we may apply the Riemann--Lebesgue lemma
  to the second integral and we have
  \begin{align}
    \frac{1}{\pi}
    \int_0^\delta
    \frac{\omega(\lambda)\sin\lambda x \cos\lambda y}
    {\abs{\varPsi(\lambda)}^2}
    \, \cd \lambda
    \longrightarrow 0,
    \quad \text{as \(y\to\pm\infty\).}
  \end{align}
  For the first integral of~\eqref{eq:h-property-finvar}, we use
  Lemma~\ref{Lem:psi-conv} and
  Lemma~\ref{Lem:JT-DI} below
  to show that
  \begin{align}
    \frac{1}{\pi}
    \int_0^\delta
    \frac{\theta(\lambda)\sin\lambda x \sin\lambda y}
    {\abs{\varPsi(\lambda)}^2}
    \, \cd \lambda
     & =
    \frac{1}{\pi}
    \int_0^\delta
    \abs*{\frac{\lambda^2}{\varPsi(\lambda)}}^2
    \frac{\theta(\lambda)}{\lambda^2}
    \frac{\sin \lambda x}{\lambda}
    \frac{\sin \lambda y}{\lambda}
    \, \cd \lambda            \\
     & \to \pm \frac{x}{m^2},
    \quad \text{as \(y\to\pm\infty\).}
  \end{align}
  This ends the proof.
\end{proof}
The following lemma is an elementary calculus.
\begin{Lem}[Jordan's theorem for the Dirichlet integral]\label{Lem:JT-DI}
  Let \(\delta>0\) and let
  \(f\colon (0,\delta)\to \bR\) be continuous and be of bounded variation.
  Then it holds that
  \begin{align}
    \lim_{x\to \pm\infty}\frac{2}{\pi} \int_0^\delta
    f(\lambda)\frac{\sin \lambda x}{\lambda} \, \cd \lambda=\pm f(0+),
  \end{align}
  where \(f(0+)=\lim_{\lambda\to 0+}f(\lambda)\).
\end{Lem}
\begin{proof}
  By integration by parts, we have
  \begin{align}
    \frac{2}{\pi} \int_0^\delta
    f(\lambda)\frac{\sin \lambda x}{\lambda} \, \cd \lambda
    &= \frac{2}{\pi}  \int_0^\delta \rbra*{\int_0^\lambda
    \cd f( \xi)}\frac{\sin \lambda x}{\lambda}\, \cd \lambda
    +\frac{2}{\pi} f(0+) \int_0^\delta \frac{\sin \lambda x}{\lambda}\, \cd \lambda\\
    &=\frac{2}{\pi}\int_0^\delta \rbra*{\int_\xi^\delta
    \frac{\sin \lambda x}{\lambda}\, \cd \lambda} \cd f(\xi)
    +\frac{2}{\pi} f(0+) \int_0^\delta \frac{\sin \lambda x}{\lambda}\, \cd \lambda\\
    &\to 0 \pm f(0+), \qquad\text{as \(x\to\pm \infty\).}
  \end{align}
  Thus we obtain the desired result.
\end{proof}

Then let us prove the existence of \(h\)
in the case \(X\) is recurrent
and \(m^2=\infty\). Its proof
is quite different from that in the case \(m^2<\infty\).

\begin{proof}[Proof of~(\ref{Thm-item:exist-h}) of Theorem~\ref{Thm:exist-h}
    in the case \(m^2=\infty\)]
  Since \(h_q(0) = 0\),
  we have the limit \(h(0)=\lim_{q\to 0+}h_q(0)=0\).

  Fix \(a \ne 0\) and set
  \begin{align}
    \overline{h}(a) = \limsup_{q\to 0+}h_q(a), \quad
    \underline{h}(a) = \liminf_{q\to 0+}h_q(a).
  \end{align}
  We also define
  \(\varDelta = \overline{h}(a) - \underline{h}(a) \ge 0\) and
  \(A = \{2^j a\colon j=0,1,2,3,\ldots\}\).

  It follows from
  \(h_q(x) \le h_q(x)+h_q(-x)\)
  and~\eqref{eq:h_q^s}--\eqref{eq:hs-conv-DCT}
  that \({\{h_q(x)\}}_{q>0}\) is bounded for each \(x\in \bR\).
  Hence, by the diagonal argument,
  we can take two
  sequences \(\{q_n\}, \{q^\prime_n\}\),
  which satisfies the following three conditions:
  \begin{itemize}
    \item \( q_n, q^\prime_n \to 0+\) as \(n\to\infty\);
    \item \(\lim_{n\to\infty}h_{q_n}(x)\) and
          \(\lim_{n\to\infty}h_{q^\prime_n}(x)\) exist and are finite for
          each \(x \in A\);
    \item \( \overline{h}(a) = \lim_{n\to\infty} h_{q_n}(a)\) and
          \(\underline{h}(a) = \lim_{n\to\infty} h_{q^\prime_n}(a)\).
  \end{itemize}
  Then we define, for \(x \in A\),
  \begin{align}
    \overline{h}(x) = \lim_{n\to\infty}h_{q_n}(x),
    \quad
    \underline{h}(x) = \lim_{n\to\infty} h_{q^\prime_n}(x).
  \end{align}
  By (\ref{Lem-item:h^D}) of Lemma~\ref{Lem:exist-hheq}, we have
  \begin{align}
    h^D(x,0)=\overline{h}(2x) - 2 \overline{h}(x)
    =\underline{h}(2x) - 2 \underline{h}(x),
    \quad x \in A.
  \end{align}
  This implies that
  \(\overline{h}(2x) - \underline{h}(2x)
  =2(\overline{h}(x) - \underline{h}(x))\).
  Hence it holds that
  \(\overline{h}(2^j a) - \underline{h}(2^j a)
  = 2^j \varDelta\), i.e.,
  \begin{align}\label{eq:diff-overunder-h}
    \frac{\overline{h}(2^j a)}{2^j a} -
    \frac{\underline{h}(2^j a)}{2^j a}
    = \frac{\varDelta}{a}, \quad j = 0,1,2,3,\ldots.
  \end{align}
  It follows from (\ref{Lem-item:h^S/x}) of Lemma~\ref{Lem:exist-hheq}
  that
  \begin{align}
    \max\cbra*{\frac{\overline{h}(2^j a)}{2^j a},
      \frac{\underline{h}(2^j a)}{2^j a}}
    \le
    \frac{h^S(2^j a)}{2^j a} \longrightarrow 0,
    \quad \text{as \(j\to\infty\).}
  \end{align}
  Thus, by letting \(j\to\infty\) in~\eqref{eq:diff-overunder-h},
  we obtain \(\varDelta = 0\).
  Therefore, we conclude that \(h(a)=\lim_{q\to 0+}h_q(a)\) exists.
\end{proof}

Next, we prove (\ref{Thm-item:unif-conv-h}) and (\ref{Thm-item:subadditive-h})
of Theorem~\ref{Thm:exist-h}
in both cases
\(m^2=\infty\) and \(m^2<\infty\).
\begin{proof}[Proof of (\ref{Thm-item:unif-conv-h}) and (\ref{Thm-item:subadditive-h})
    of Theorem~\ref{Thm:exist-h}]
  By the Markov property, we have, for \(x,y\in\bR\),
  \begin{align}
    \bP_{x+y}\sbra{\ce^{-qT_0}} =\bP_0\sbra{\ce^{-qT_{-x-y}}}
    \ge \bP_0\sbra{\ce^{-qT_{-x}} \bP_{-x}\sbra{\ce^{-qT_{-x-y}}}}
    = \bP_x\sbra{\ce^{-qT_{0}}} \bP_y\sbra{\ce^{-qT_0}}.
  \end{align}
  Since \(h_q(x) =r_q(0)\rbra{1-\bP_x\sbra{e^{-qT_0}}}\)
  and \(\rbra{1-\bP_x\sbra{e^{-qT_0}}}\rbra{1-\bP_y\sbra{e^{-qT_0}}}\ge 0\),
  it holds that
  \begin{align}
    h_q(x+y) \le h_q(x) +h_q(y).
  \end{align}
  Hence \(h_q\) is subadditive.
  (This proof can also be found in~\cite[Lemma 3.3]{MR3689384}.)
  Since \(h_q\) is non-negative and subadditive,
  and by~\eqref{eq:h_q^s} and~\eqref{eq:hs-conv-DCT},
  it holds that
  \begin{align}
    \abs{h_q(x+\delta) - h_q(x)}
    \le h_q(\delta) + h_q(-\delta)
    \le \int_0^\infty
    \abs*{\frac{\rbra{\lambda \delta}^2 \wedge 2}{\varPsi(\lambda)}}
    \, \cd \lambda.
  \end{align}
  Hence \({\{h_q\}}_{q>0}\) is equi-continuous.
  Equi-continuity and pointwise-convergence imply
  the uniform convergence on compact subset of \(\bR\).
  The subadditivity of \(h\) follows directly from that of \(h_q\).
\end{proof}
Finally, we show the properties of \(h\) in Theorem~\ref{Thm:property-h}
in the case \(m^2=\infty\).
\begin{proof}[Proof of Theorem~\ref{Thm:property-h}
    in the case \(m^2=\infty\)]
  \noindent (\ref{Thm-item:h/x-infinity})
  This is directly from
  (\ref{Lem-item:h^S/x}) of Lemma~\ref{Lem:exist-hheq}.

  \noindent (\ref{Thm-item:h-h-limit})
  Since \(h\) is subadditive, we have
  \begin{align}\label{eq:h-diff-sum}
    \sum_{k=1}^n \cbra{h(kx+y)-h((k-1)x+y)}
    = h(nx+y) - h(y) \le h(nx).
  \end{align}
  By (\ref{Lem-item:h^D}) of Lemma~\ref{Lem:exist-hheq},
  it holds that
  \begin{align}
    \cbra{h(2x+y) - h(x+y)} - \cbra{h(x+y) - h(y)}
    \longrightarrow 0,
    \quad \text{as \(y\to\pm\infty\).}
  \end{align}
  Thus we have
  \begin{align}\label{eq:h-diff-limsup}
    \limsup_{y\to\pm\infty}
    \sum_{k=1}^n \cbra{h(kx+y)-h((k-1)x+y)}
    = n \limsup_{y\to\pm\infty} \cbra{h(x+y) - h(y)}.
  \end{align}
  Combining~\eqref{eq:h-diff-sum} and~\eqref{eq:h-diff-limsup}, we obtain
  \begin{align}
    \limsup_{y\to\pm\infty} \cbra{h(x+y) - h(y)}
    \le \frac{h(nx)}{n}.
  \end{align}
  Since we have \(\lim_{n\to\infty}\frac{h(nx)}{n}=0\)
  by~(\ref{Lem-item:h^D}) of Lemma~\ref{Lem:exist-hheq}, we have
  \begin{align}
    \limsup_{y\to\pm\infty} \cbra{h(x+y) - h(y)} \le 0.
  \end{align}
  Replacing \(x\) with \(-x\), we also have
  \begin{align}
    \liminf_{y\to\pm\infty} \cbra{h(y) - h(y-x)} \ge 0.
  \end{align}
  Therefore we obtain
  \(
  \lim_{y\to\pm\infty} \cbra{h(x+y) - h(y)} = 0.
  \)
\end{proof}

\subsection{The function \(h^B\)}
Let us compute \(\bP\sbra{L_{T_a}}\) and \(\bP_x(T_a<T_b)\).
\begin{Lem}\label{Lem:h^B-hitting-h}
  \begin{enumerate}
    \item For \(a \in \bR\),
          \begin{align}
            h^B_q(a) \coloneqq
            \bP\sbra*{\int_0^{T_a} \ce^{-qt}\, \cd L_t}
            = h_q(a)+h_q(-a)-\frac{h_q(a)h_q(-a)}{r_q(0)}.
            \label{eq:LT-hitting-expect-q}
          \end{align}\label{Lem-item:LT-hitting-expect}
          Consequently, it holds that
          \begin{align}
            h^B(a) \coloneqq \lim_{q\to 0+} h_q^B(a)
            =\bP\sbra{L_{T_a}}= h(a) + h(-a).
            \label{eq:h^B}
          \end{align}\label{Lem-item:hq-two-hitting-time}
    \item For \(x, a, b \in \bR\), \(a\ne b\) and
          for \(q>0\), it holds that
          \begin{align}
            \begin{aligned}
               & \bP_x\sbra{\ce^{-qT_a}; T_a < T_b}
              \\
               & =
              \frac{h_q(b-a)+h_q(x-b)-h_q(x-a)-h_q(x-b)h_q(b-a)/r_q(0)}{h^B_q(a-b)}.
            \end{aligned}
            \label{eq:qT_a;T_a<T_b}
          \end{align}
          Consequently, it holds that
          \begin{align}
            \bP_x(T_a < T_b)
            = \frac{h(b-a)+h(x-b)-h(x-a)}{h^B(a-b)}.
            \label{eq:two-hittig-time-prob}
          \end{align}\label{Lem-item:two-hittig-time-prob}
  \end{enumerate}
\end{Lem}

\begin{proof}
  \noindent (\ref{Lem-item:hq-two-hitting-time})
  We omit the proof of~\eqref{eq:LT-hitting-expect-q},
  which can be found
  in~\cite[Lemma V.11]{MR1406564}.
  Letting \(q\to 0+\) in~\eqref{eq:LT-hitting-expect-q},
  and using~\eqref{eq:kappa},
  we obtain~\eqref{eq:h^B}.

  \noindent (\ref{Lem-item:two-hittig-time-prob})
  By the strong Markov property, it holds that,
  for \(x,a,b \in \bR,\,a\ne b, \, q>0\),
  \begin{gather}
    \bP_x\sbra{\ce^{-qT_a}}
    = \bP_x\sbra{\ce^{-qT_a};T_a < T_b}
    + \bP_x\sbra{\ce^{-qT_b}; T_b < T_a}\bP_b\sbra{\ce^{-qT_a}}, \\
    \bP_x\sbra{\ce^{-qT_b}}
    = \bP_x\sbra{\ce^{-qT_b};T_b < T_a}
    + \bP_x\sbra{\ce^{-qT_a}; T_a < T_b}\bP_a\sbra{\ce^{-qT_b}}.
  \end{gather}
  Combining the above two equalities, we have
  \begin{align}
    \bP_x\sbra{\ce^{-qT_a}; T_a<T_b}
    = \frac{\bP_x\sbra{\ce^{-qT_a}}-\bP_x\sbra{\ce^{-qT_b}}
      \bP_b\sbra{\ce^{-qT_a}}}
    {1-\bP_a\sbra{\ce^{-qT_b}}\bP_b\sbra{\ce^{-qT_a}}}.
  \end{align}
  By~\eqref{eq:exp-hit-time}, this implies that
  \begin{align}
     & \bP_x\sbra{\ce^{-qT_a}; T_a<T_b}
    \\
     & = \frac{r_q(a-x) - r_q(b-x)r_q(a-b)/r_q(0)}
    {r_q(0) - r_q(b-a)r_q(a-b)/r_q(0)}
    \\
     & = \frac{h_q(b-a)+h_q(x-b)-h_q(x-a)-h_q(x-b)h_q(b-a)/r_q(0)}{h^B_q(a-b)}.
  \end{align}
  Hence we obtain~\eqref{eq:qT_a;T_a<T_b}.
  Letting \(q\to 0+\) in~\eqref{eq:qT_a;T_a<T_b},
  we obtain~\eqref{eq:two-hittig-time-prob}.
\end{proof}
\begin{Rem}
  The formulae~\eqref{eq:qT_a;T_a<T_b} and~\eqref{eq:two-hittig-time-prob} are
  also discussed in Theorem 6.5
  of Getoor~\cite{MR185663}.
  See also Proposition 5.3, Proposition 5.4
  and Remark 5.5 of Yano--Yano--Yor~\cite{MR2599211}.
\end{Rem}
The next theorem can be proved in the same way as
Pant\'{\i}~\cite[Lemma 3.10]{MR3689384}.

\begin{Lem}[{\cite[Lemma 3.10]{MR3689384}}]\label{Thm:h^B-infty}
  It holds that \(\lim_{x\to\infty} h^B(x) = \infty\).
\end{Lem}
\begin{proof}
  For completeness of this paper, we give the proof.
  Let \(\bm{e}\) be an independent exponential time of mean \(1\)
  and set \(\bm{e}_q \coloneqq \bm{e}/q\) for \(q>0\).
  Then we have
  \begin{align}\label{eq:h_b-eq}
    h^B(x) =\bP\sbra{L_{T_x}}
    = \bP\sbra{L_{T_x}; T_x\le \bm{e}_q} + \bP\sbra{L_{T_x}; T_x > \bm{e}_q}
    \ge \bP\sbra{L_{\bm{e}_q}; T_x>\bm{e}_q}.
  \end{align}
  Letting \(x\to\infty\), we have
  \begin{align}\label{eq:h_b-lower}
    \liminf_{x\to\infty}h^B(x) \ge \bP\sbra{L_{\bm{e}_q}} = r_q(0),
    \quad \text{for all \(q>0\).}
  \end{align}
  Since \(X\) is recurrent, i.e.,
  \(\kappa=0\) in~\eqref{eq:kappa},
  we let \(q\to 0+\) to obtain
  \( \liminf_{x\to\infty}h^B(x) \ge \infty\).
  Hence we obtain the desired result.
\end{proof}

The following theorem, which will be used in Section~\ref{Sec:inv-time},
is a generalization of the result
in the symmetric case
by Yano~\cite[Theorem 6.1]{MR2603019}.
\begin{Thm}\label{Thm:n-two-hitting-time}
  For \(a \in \bR \setminus \{0\}\), it holds that
  \begin{align}\label{eq:n-T_a-T_0-inf}
    n(T_a < T_0) = \frac{1}{h^B(a)}.
  \end{align}
\end{Thm}
\begin{proof}
  For \(l > 0\), it holds that
  \begin{align}
    \bP(L_{T_a}>l) = \bP(T_a > \eta_l) =
    \bP(\sigma_{\cbra{T_a<T_0}}> l),
  \end{align}
  where \(\sigma_A = \inf\cbra{l \colon e_l \in A}\)
  for \(A \subset \cD\).
  Since \(\sigma_A\) is the hitting time
  of the set \(A\)
  for the
  Poisson point process \(((l, e_l), l\ge 0)\),
  we have
  \begin{align}\label{eq:L_T_a-exp-distrib}
    \bP(L_{T_a}>l)
    = \ce^{-l{n(T_a<T_0)}}.
  \end{align}
  For more details, see, e.g.,~\cite[Lemma 6.17]{MR3155252}.
  In particular, \(L_{T_a}\) is exponentially distributed.
  On the other hand, we know \(h^B(a) = \bP\sbra{L_{T_a}}\)
  by Lemma~\ref{Lem:h^B-hitting-h}.
  Hence we obtain~\eqref{eq:n-T_a-T_0-inf}.
\end{proof}

\subsection{Examples of the renormalized zero resolvent}
\begin{Eg}[Brownian motion]\label{Eg:h-brownian}
  Assume \(X\) is a standard Brownian motion.
  Since \(m^2=\sigma^2=1\), it holds that
  \begin{align}
    h^{(\gamma)}(x) = \abs{x}+\gamma x=\begin{cases}
      (1+\gamma)x & x\ge 0, \\
      (1-\gamma)x & x\le 0,
    \end{cases}
    \quad\text{for \(-1\le \gamma\le 1\).}
  \end{align}
\end{Eg}
\begin{Eg}[Strictly stable process]
  Assume that \(X\) is a strictly stable process of index
  \(\alpha\in (1,2)\) with the L\'{e}vy measure
  \begin{align}
    \nu(\cd x) =
    \begin{cases}
      c_{+} \abs{x}^{-\alpha-1}\, \cd x & \text{on \((0, \infty)\),}  \\
      c_{-} \abs{x}^{-\alpha-1}\, \cd x & \text{on \((-\infty, 0)\),}
    \end{cases}
  \end{align}
  where \(c_{+},c_{-} \ge 0\) and \(c_{+}+c_{-} > 0\).
  Then its characteristic exponent is given by
  \begin{align}
    \varPsi(\lambda)
    = c\abs{\lambda}^{\alpha}
    \rbra*{1-\ci \beta \sgn(\lambda)\tan\frac{\alpha\pi}{2}},
  \end{align}
  where \(c\) and \(\beta\) are constants defined by
  \begin{align}
    c = \frac{(c_{+}+c_{-})\pi}{2\alpha \Gamma(\alpha)\sin(\pi\alpha/2)},
    \quad
    \beta = \frac{c_{+} - c_{-}}{c_{+}+c_{-}}.
  \end{align}
  In this case, we have \(m^2=\infty\) and the function
  \(h\) can be represented as
  \begin{align}\label{eq:stable-h}
    h(x) = \frac{1}{K(\alpha)} (1-\beta\sgn (x))\abs{x}^{\alpha-1},
  \end{align}
  where
  \begin{align}
    K(\alpha) = -2c\Gamma(\alpha)\cos\frac{\pi\alpha}{2}
    \rbra*{1+\beta^2\tan^2 \frac{\pi\alpha}{2}}.
  \end{align}
  For more details, see~\cite[Section 5]{MR3072331}.
\end{Eg}

\section{Local time penalization with exponential clock}\label{Sec:exp-clock}
We now start to deal with the penalization
result with exponential clock.
Let \(\bm{e}\) be an independent exponential time of mean
\(1\) and for \(q>0\), we write \(\bm{e}_q \coloneqq \bm{e}/q\), which
has an exponential distribution of mean \(1/q\).
We compute \(\bP_x\sbra{f(L_{\bm{e}_q})}\),
\(\bP_x\sbra{f(L_{\bm{e}_q})|\cF_t}\) and its limit as \(q\to 0+\) to
investigate \(\lim_{q\to 0+}\bP_x\sbra{F_t
  f(L_{\bm{e}_q})}/\bP_x\sbra{f(L_{\bm{e}_q})}\) for
bounded \(\cF_t\)-measurable functional \(F_t\).
Recall that we assume \(X\) is recurrent
and assume the condition~\ref{item:assumption}.
\subsection{The law of the local time with exponential clock}
First, we compute \(\bP_x\sbra{f(L_{\bm{e}_q})}\).
\begin{Lem}\label{Lem:expect-L_e_q}
  Let \(f\) be a non-negative measurable function.
  Then, for \(q>0\) and \(x \in \mathbb{R}\), it holds that
  \begin{align}\label{eq:lem-expect-L_e_q}
    \bP_x\sbra{f(L_{\bm{e}_q})}
    = \frac{1}{r_q(0)}\cbra*{h_q(x)f(0)
    + \rbra*{1-\frac{h_q(x)}{r_q(0)}}
    \int_0^\infty \ce^{-u/{r_q(0)}}f(u) \, \cd u}.
  \end{align}
\end{Lem}
\begin{proof}
  Using the excursion theory, we have
  \begin{align}
    \bP_0\sbra*{\int_0^\infty f(L_t) q \ce^{-qt}\, \cd t}
     & = \bP_0\sbra*{\sum_{u\in D}
    \int_{\eta_{u-}}^{\eta_u} f(u)q\ce^{-qu}\, \cd u} \\
     & = \bP_0 \otimes \widetilde{n}
    \sbra*{\int_0^\infty \cd L_t \, f(L_t) \ce^{-qt}
      \int_0^{\widetilde{T}_0} \cd u \, q\ce^{-qu}},
  \end{align}
  where the last equality follows from Lemma~\ref{Lem:compensation-formula}.
  By~\eqref{eq:exp-nt0} and~\eqref{eq:rel-n-r_q}, we have
  \begin{align}
    \bP_0 \otimes \widetilde{n}
    \sbra*{\int_0^\infty \cd L_t \, f(L_t) \ce^{-qt}
      \int_0^{\widetilde{T}_0} \cd u \, q\ce^{-qu}}
     & = \bP_0\sbra*{\int_0^\infty \cd u \, f(u) \ce^{-q\eta_u}}
    n\sbra*{1 - \ce^{-qT_0}}                                            \\
     & = \frac{1}{r_q(0)}\int_0^\infty f(u) \ce^{-u/{r_q(0)}} \, \cd u.
  \end{align}
  Applying the Markov property, we obtain
  \begin{align}
    \bP_x\sbra{f(L_{\bm{e}_q})}
     & = \bP_x \sbra*{\int_0^\infty f(L_t)q\ce^{-qt}\, \cd t}                    \\
     & = \bP_x\sbra*{\int_0^{T_0} f(0) q\ce^{-qt}\, \cd t}
    + \bP_x\sbra{\ce^{-qT_0}}\bP_0\sbra*{\int_0^\infty f(L_t)q\ce^{-qt}\, \cd t} \\
     & = f(0)\rbra*{1-\frac{r_q(-x)}{r_q(0)}}
    + \frac{r_q(-x)}{r_q(0)} \int_0^\infty f(u) \ce^{-u/{r_q(0)}}\, \cd u        \\
     & = \frac{1}{r_q(0)}\cbra*{h_q(x)f(0)
    + \rbra*{1-\frac{h_q(x)}{r_q(0)}}\int_0^\infty \ce^{-u/{r_q(0)}}f(u) \, \cd u},
  \end{align}
  here we used~\eqref{eq:exp-hit-time}.
  Therefore we obtain the desired result.
\end{proof}

\subsection{A.s.\ convergence for exponential clock}
To calculate \(\bP_x\sbra{F(L_{\bm{e}_q})|\cF_t}\),
we separate into the two cases \(\cbra{t<\bm{e}_q }\) and \(\cbra{\bm{e}_q\le t}\).

Let \(f \in \cL_{+}^1\) and \(x \in \bR\). For \(q > 0\),
define
\begin{align}
  N_t^{q} & = r_q(0) \bP_x\sbra{f(L_{\bm{e}_q}); t < \bm{e}_q | \cF_t}, \\
  M_t^{q} & = r_q(0) \bP_x\sbra{f(L_{\bm{e}_q}) | \cF_t},               \\
  A_t^{q} & = M_t^{q} - N_t^{q}
  = r_q(0) \bP_x\sbra{f(L_{\bm{e}_q}); \bm{e}_q \le t | \cF_t},
\end{align}
and \(h^{(\gamma)}\) and \(M_t^{(\gamma)}\) are defined in~\eqref{eq:def-h-gamma}
and~\eqref{eq:def-mart}.

\begin{Thm}
  For \(f \in \cL_{+}^1\) and \(x \in \bR\), it holds that
  \begin{align}
    \lim_{q\to 0+}N_t^{q}
    =\lim_{q\to 0+} M_t^{q}
    = M_t^{(0)}, \qquad \bP_x\text{-a.s.}
  \end{align}
\end{Thm}
\begin{proof}
  By the Markov property and the additivity of \(L\),
  we have
  \begin{align}
    N_t^{q}
     & = r_q(0) \bP_x\sbra{f(L_{\bm{e}_q}); t < \bm{e}_q | \cF_t}
    \\
     & = r_q(0) \ce^{-qt} \widetilde{\bP}_{X_t}
    \sbra{f(L_t+\widetilde{L}_{\widetilde{\bm{e}}_q})}
    \\
     & = \ce^{-qt}
    \cbra*{h_q(X_t)f(L_t) +
      \rbra*{1-\frac{h_q(X_t)}{r_q(0)}}
      \int_0^\infty \ce^{-u/r_q(0)} f(L_t + u) \, \cd u},
  \end{align}
  here the last equality, we used Lemma~\ref{Lem:expect-L_e_q}.
  Since \(1-\frac{h_q(X_t)}{r_q(0)}=\bP_{X_t}\sbra{\ce^{-qT_0}} \to 1\),
  \(\bP_x\)-a.s.\ as \(q\to 0+\) and
  since \(\int_0^\infty f(L_t+u)\, \cd u<\infty\),
  we may apply the dominated convergence theorem to deduce that
  \(N_t^{q} \longrightarrow M_t^{(0)}\),
  \(\bP_x\)-a.s.\ as \(q\to 0+\).
  By (\ref{Lem-item:qr_q}) of Lemma~\ref{Lem:inte-of-varPsi}, we have
  \begin{align}
    A_t^{q} & = r_q(0) \bP_x\sbra{f(L_{\bm{e}_q}); \bm{e}_q \le t | \cF_t}
    = qr_q(0) \int_0^t f(L_u) \ce^{-qu} \, \cd u
    \to 0
  \end{align}
  \(\bP_x\)-a.s.\ as \(q \to 0+\).
  Therefore, we obtain \(M_t^{q} \to M_t^{(0)}\),
  \(\bP_x\)-a.s.\ as \(q \to 0+\).
\end{proof}

\subsection{\(\cL^1\) convergence for exponential clock}
Now we prepare some lemma to prove the \(\cL^1\)
convergence for exponential clock.
The following lemma is a part of
Theorem 15.2 of Tsukada~\cite{MR3838874}.
\begin{Lem}[{\cite[Theorem 15.2]{MR3838874}}]\label{Lem:h-l1-conv}
  For \(t \ge 0\), it holds that
  \(
  h_q(X_t) \longrightarrow h(X_t)
  \)
  in \(\cL^1(\bP_x)\) as \(q \to 0+\).
\end{Lem}
The next theorem is the penalization result with exponential clock.
\begin{Thm}\label{Thm:exp-L^1-conv}\label{Thm:exp-time-result}
  Let \(f \in \cL_{+}^1\) and \(x \in \bR\). Then
  \(( M_t^{(0)}, t \ge 0)\) is a non-negative
  \(((\cF_t), \bP_x)\)-martingale,
  and it holds that
  \begin{align}
    \lim_{q\to 0+}N_t^{q}
    = \lim_{q\to 0+}M_t^{q}
    = M_t^{(0)},
    \qquad \text{in \(\cL^1(\bP_x)\).}
  \end{align}
  Consequently, if \(M_0^{(0)}>0\) under \(\bP_x\),
  it holds that
  \begin{align}\label{eq:exp-penalized-meas}
    \frac{\bP_x\sbra{F_t f(L_{\bm{e}_q})}}{\bP_x\sbra{f(L_{\bm{e}_q})}}
    \longrightarrow \bP_x\sbra*{F_t \frac{M_t^{(0)}}{M_0^{(0)}}},
    \qquad \text{as \(q \to 0+\),}
  \end{align}
  for all bounded \(\cF_t\)-measurable functionals \(F_t\).
\end{Thm}
Note that the penalized measure in~\eqref{eq:exp-penalized-meas}
is not the same as that of Theorems~\ref{Thm:hitting-time-result}
and~\ref{Thm:two-point-hitting-time-result}.
\begin{proof}[Proof of Theorem~\ref{Thm:exp-L^1-conv}]
  We first consider the case where \(f\) is bounded.
  We write
  \begin{align}
    N_t^{q}
     & = \ce^{-qt}\cbra*{h_q(X_t)f(L_t) +
      \rbra*{1-\frac{h_q(X_t)}{r_q(0)}}
    \int_0^\infty \ce^{-u/r_q(0)} f(L_t + u) \, \cd u} \\
     & \eqqcolon {(\mathrm{I})}_q + {(\mathrm{II})}_q, \\
    M_t^{(0)}
     & = h(X_t)f(L_t) +
    \int_0^\infty
    f(L_t + u) \, \cd u                                \\
     & \eqqcolon {(\mathrm{I})} + {(\mathrm{II})}.
  \end{align}
  By Lemma~\ref{Lem:h-l1-conv}
  and by the boundedness of \(f\),
  we obtain \({(\mathrm{I})}_q \to (\mathrm{I})\) in \(\cL^1(\bP_x)\).
  Moreover, since \(\int_0^\infty f(u)\, \cd u<\infty\),
  it follows from the dominated convergence theorem that
  \({(\mathrm{II})}_q \to (\mathrm{II})\) in \(\cL^1(\bP_x)\).
  Hence we obtain \(N_t^{q} \to M_t^{(0)}\) in \(\cL^1(\bP_x)\).
  By (\ref{Lem-item:qr_q}) of Lemma~\ref{Lem:inte-of-varPsi},
  we have
  \begin{align}
    \bP_x\sbra{A_t^{q}}
    = q r_q(0) \bP_x\sbra*{\int_0^t \ce^{-qu} f(L_u)\ \cd u}
    \le q r_q(0) t\norm{f}
    \to 0, \qquad \text{as } q \to {0+}.
  \end{align}
  Since \(A_t^q\ge 0\),
  this means that \(A_t^q\to 0\) in \(\cL^1(\bP_x)\). Thus we have
  \(M_t^{q} \to M_t^{(0)}\) in \(\cL^1(\bP_x)\).
  For \(0 \le s \le t\), we know \(\bP_x\sbra{M_t^{q}|\cF_s} = M_s^{q}\).
  Letting \(q \to 0+\) on both sides, we have
  \begin{align}\label{eq:exp-mart}
    \bP_x\sbra{M_t^{(0)}|\cF_s} = M_s^{(0)},
  \end{align}
  which means that \((M_t^{(0)},t\ge 0)\) is a non-negative
  \(((\cF_t), \bP_x)\)-martingale.\

  Let us consider the general \(f \in \cL_{+}^1\).
  We know the equality~\eqref{eq:exp-mart} holds for \(f \wedge n\).
  Letting \(n \to \infty\),~\eqref{eq:exp-mart}
  holds for general \(f \in \cL_{+}^1\)
  by the monotone convergence theorem.
  Hence we have
  \begin{align}
    \lim_{q \to 0+}\bP_x\sbra{M_t^{q}} = \lim_{q \to 0+}M_0^{q}
    = M_0^{(0)} = \bP_x\sbra{M_t^{(0)}},
  \end{align}
  and by Fatou's lemma,
  \begin{align}
    M_0^{(0)}
    = \lim_{q \to 0+}\bP_x\sbra{M_t^{q}}
    \ge \limsup_{q \to 0+}\bP_x\sbra{N_t^{q}}
    \ge \liminf_{q \to 0+}\bP_x\sbra{N_t^{q}}
    \ge M_0^{(0)}.
  \end{align}
  Thus we have
  \(\lim_{q\to 0+}\bP_x\sbra{N_t^{q}}=\lim_{q\to 0+}\bP_x\sbra{M_t^{q}}
  =\bP_x\sbra{M_t^{(0)}}\).
  Applying Scheff\'{e}'s lemma, we obtain
  \(\lim_{q\to 0+}N_t^{q}=\lim_{q\to 0+} M_t^{q} = M^{(0)}_t\) in \(\cL^1(\bP_x)\).
\end{proof}

\section{Local time penalization with hitting time clock}\label{Sec:hitting-time}
We deal with the penalization result with hitting time clock \((T_a)\).
To this aim, we compute \(\bP_x\sbra{L_{T_a}}\)
and \(\bP_x\sbra{L_{T_a}|\cF_t}\) and its limit as \(a\to\pm\infty\).
Recall that we assume \(X\) is recurrent
and assume the condition~\ref{item:assumption}.
\subsection{The law of the local time with hitting time clock}
First, we compute \(\bP_x\sbra{f(L_{T_a})}\).
\begin{Lem}\label{Lem:expect-L_T_a}
  For \(x, a \in \bR\), \(a\ne 0\) and any non-negative measurable
  function \(f\), we have
  \begin{align}
    \bP_x\sbra{f(L_{T_a})}
    = \bP_x(T_0 > T_a) f(0)
    + \frac{\bP_x(T_0<T_a) }{h^B(a)}
    \int_0^\infty \ce^{-u/h^B(a)} f(u) \, \cd u.
  \end{align}
\end{Lem}
\begin{proof}
  By~\eqref{eq:L_T_a-exp-distrib},
  \(L_{T_a}\) is exponentially distributed with
  parameter \(1/h^B(a)\).
  Hence we have
  \begin{align}
    \bP_0\sbra{f(L_{T_a})}
    = \frac{1}{h^B(a)}
    \int_0^\infty \ce^{-u/h^B(a)} f(u) \, \cd u.
  \end{align}
  Applying the Markov property,
  we conclude that
  \begin{align}
    \bP_x\sbra{f(L_{T_a})}
     & = \bP_x( T_0 > T_a) f(0) + \bP_x(T_0 < T_a) \bP_0\sbra{f(L_{T_a})} \\
     & = \bP_x(T_0 > T_a)f(0)
    + \frac{\bP_x(T_0<T_a) }{h^B(a)}
    \int_0^\infty \ce^{-u/h^B(a)} f(u) \, \cd u.
  \end{align}
  Hence the proof is complete.
\end{proof}

\subsection{Proof of a.s.\ convergence of
  Theorem~\ref{Thm:hitting-time-result}}\label{Subsec:pf-hitting}
We now proceed the proof of the hitting time result.
We separate into the two cases \(\cbra{t<T_a}\) and \(\cbra{T_a\le t}\).
\begin{proof}[Proof of a.s.\ convergence of Theorem~\ref{Thm:hitting-time-result}]
  By the strong Markov property, the additivity of \(L\)
  and Lemma~\ref{Lem:expect-L_T_a}, we have
  \begin{align}
    N_t^a
     & =
    1_{\cbra{t<T_a}} h^B(a)
    \widetilde{\bP}_{X_t}
    \sbra{f(\widetilde{L}_{\widetilde{T}_a}+L_t)} \\
     & =
    1_{\cbra{t<T_a}}
    \cbra*{h^B(a)\bP_{X_t}(T_0 > T_a)f(L_t)
      + \bP_{X_t}(T_0<T_a)
      \int_0^\infty \ce^{-u/h^B(a)}f(L_t+u)
      \, \cd u
    },
  \end{align}
  \(\bP_x\)-a.s.\ as \(a\to\pm\infty\).
  (\ref{Thm-item:h-h-limit}) of Theorem~\ref{Thm:property-h}
  and~\eqref{eq:two-hittig-time-prob} imply that
  \begin{gather}
    h^B(a)\bP_{X_t}(T_0> T_a)
    = h(X_t)+h(-a)-h(X_t-a)
    \longrightarrow h^{(\pm 1)}(X_t) , \\
    \bP_{X_t}(T_0<T_a) \longrightarrow 1,
  \end{gather}
  \(\bP_x\)-a.s.\ as \(a\to\pm\infty\).
  By Lemma~\ref{Thm:h^B-infty}, we obtain
  \(N_t^a \to M_t^{(\pm 1)}\),
  \(\bP_x\)-a.s.\ as \(a\to\pm\infty\).
  Furthermore, we have
  \begin{align}
    M_t^a - N_t^a
     & =
    h^B(a) \bP_x\sbra*{f(L_{T_a}); T_a\le t|\cF_t}\label{eq:hitting-A} \\
     & =
    h^B(a) f(L_{T_a}) 1_{\cbra{T_a\le t}}                              \\
     & \to 0,
    \quad \bP_x\text{-a.s.\ as \(a\to\pm\infty\).}\label{eq:hitting-Al}
  \end{align}
  Hence we obtain \(M_t^a \to M_t^{(\pm 1)}\)
  \(\bP_x\)-a.s.\ as \(a\to\pm\infty\).
\end{proof}

\subsection{Proof of \(\cL^1\) convergence
  of Theorem~\ref{Thm:hitting-time-result}}\label{Subsec:pf-hitting-l1}
\begin{proof}[Proof of \(\cL^1\) convergence of Theorem~\ref{Thm:hitting-time-result}]
  We first consider the case
  \(m^2 = \infty\).
  Then we know \(M_t^{(\pm 1)} = M_t^{(0)}\).
  Hence, by Theorem~\ref{Thm:exp-L^1-conv},
  \(( M_t^{(\pm 1)}, t \ge 0)\) is a non-negative
  \(((\cF_t), \bP_x)\)-martingale.
  Thus we have
  \begin{align}
    \bP_x\sbra{M_t^{a}} = M_0^{a}
    \to M_0^{(\pm 1)} = \bP_x\sbra{M_t^{(\pm 1)}},
    \quad \text{as \(a\to\pm\infty\).}
  \end{align}
  By Fatou's lemma, we have
  \begin{align}
    \bP_x\sbra{M_t^{(\pm 1)}}
    =\lim_{a\to\pm\infty}\bP_x\sbra{M_t^{a}}
    \ge \limsup_{a\to\pm\infty}\bP_x\sbra{N_t^{a}}
    \ge \liminf_{a\to\pm\infty}\bP_x\sbra{N_t^{a}}
    \ge \bP_x\sbra{M_t^{(\pm 1)}}.
  \end{align}
  Consequently, it holds that
  \(\bP_x\sbra{N_t^{a}}, \bP_x\sbra{N_t^{a}}
  \to \bP_x\sbra{M_t^{(\pm 1)}}\), as \(a\to\pm\infty\).
  Applying Scheff\'{e}'s lemma,
  we obtain
  \(N_t^{a}, N_t^{a} \to M_t^{(\pm 1)}\)
  in \(\cL^1(\bP_x)\) as \(a\to\pm\infty\).

  We next consider the case \(m^2 < \infty\).
  Suppose first that \(f\) is bounded.
  We write
  \begin{align}
    N_t^{a}
     & =
    1_{\cbra{t<T_a}}
    \rbra{h(X_t)+h(-a)-h(X_t-a)}f(L_t)                 \\
     & \qquad +
    1_{\cbra{t<T_a}}
    \bP_{X_t}(T_0<T_a)
    \int_0^\infty \ce^{-u/h^B(a)}f(L_t+u)
    \, \cd u                                           \\
     & \eqqcolon {(\mathrm{I})}_a + {(\mathrm{II})}_a, \\
    M_t^{{(\pm 1)}}
     & = h^{(\pm 1)}(X_t) f(L_t)
    + \int_0^\infty f(L_t+u) \, \cd u\label{eq:useless-dot}         \\
     & \eqqcolon {(\mathrm{I})} + {(\mathrm{II})}.
  \end{align}
  Since \(h\) is subadditive, we have
  \(h(X_t)+h(-a)-h(X_t-a) \le h(X_t) + h(-X_t)\).
  By the proof of Lemma~\ref{Lem:h-l1-conv},
  we know \(\bP_x\sbra{h(X_t) + h(-X_t)} < \infty\)
  and thus, by the dominated convergence theorem,
  we have
  \begin{align}
    h(X_t)+h(-a)-h(X_t-a)
    \longrightarrow
    h^{(\pm 1)}(X_t),
    \quad \text{in \(\cL^1(\bP_x)\) as \(a\to\pm\infty\).}
  \end{align}
  The boundedness of \(f\) implies that
  \({(\mathrm{I})}_a \to (\mathrm{I})\)
  in \(\cL^1(\bP_x)\) as \(a\to\pm\infty\).
  Since \({(\mathrm{II})}_a \le \int_0^\infty f(u)\, \cd u\),
  we may apply the dominated convergence theorem to conclude
  \({(\mathrm{II})}_a \to (\mathrm{II})\)
  in \(\cL^1(\bP_x)\) as \(a\to\pm\infty\).
  Hence we obtain \(N_t^{a}\to M_t^{(\pm 1)}\)
  in \(\cL^1(\bP_x)\).
  By Theorem~\ref{Thm:property-h}
  and the optional stopping theorem,
  we obtain
  \begin{align}
    \bP_x\sbra{A_t^{a}}
     & =
    h^B(a) \bP_x\sbra{f(L_{T_a}); T_a \le t}    \\
     & =
    \frac{h^B(a)}{a}\frac{a}{h(a)}
    \bP_x\sbra{h(X_{T_a})f(L_{T_a}); T_a \le t} \\
     & \le
    \frac{h^B(a)}{a}\frac{a}{h(a)}
    \bP_x\sbra{M_{T_a}^{(0)}; T_a \le t}        \\
     & =
    \frac{h^B(a)}{a}\frac{a}{h(a)}
    \bP_x\sbra{M_t^{(0)}; T_a \le t}            \\
     & \to 0,
    \qquad \text{as \(a\to\pm\infty\).}
  \end{align}
  This yields that \(A_t^{a} \to 0\) in \(\cL^1(\bP_x)\) as \(a\to\pm\infty\).
  Hence \(M_t^{a} \to  M_t^{(\pm 1)}\)
  in \(\cL^1(\bP_x)\) as \(a\to\pm\infty\).
  Since \(\bP_x\sbra{M_t^a|\cF_s}=M_s^a\) for \(0\le s\le t\),
  we let \(a\to\pm\infty\) to obtain
  \begin{align}
    \bP_x\sbra{M_t^{(\pm 1)}|\cF_s}=M_s^{(\pm 1)}.\label{eq:M^1-mart}
  \end{align}
  In particular, \(( M_t^{(\pm 1)}, t \ge 0)\) is a non-negative
  \(((\cF_t), \bP_x)\)-martingale.
  To remove the boundedness condition of \(f\), we consider \(f\wedge n\) and
  let \(n\to\infty\) in~\eqref{eq:M^1-mart}. Then we have
  \(( M_t^{(\pm 1)}, t \ge 0)\) is a non-negative
  \(((\cF_t), \bP_x)\)-martingale.
  The remainder of the proof is the same as that of
  Theorem~\ref{Thm:exp-L^1-conv}. So we omit it.
\end{proof}

\begin{Rem}\label{Rem:martingale}
  Suppose that \(m^2 < \infty\).
  Then, since \(M_t^{(0)}\) and \(M_t^{(1)}\)
  are different \(((\cF_t), \bP_x)\)-martingales,
  \(
  m^2 \rbra{M_t^{(1)} - M_t^{(0)}}
  = X_t f(L_t)
  \)
  is also a \(((\cF_t), \bP_x)\)-martingale.
  In fact, if \(X_1\) is integrable, \(\bP\sbra{X_1} = 0\) and
  \(f\) is a bounded measurable function, then
  \((X_t f(L_t),t\ge 0)\) is a \(((\cF_t), \bP_x)\)-martingale.
  For more details, see Theorem~\ref{Thm:mart-XfL}.
\end{Rem}

By Remark~\ref{Rem:martingale}, we can prove Theorem~\ref{Thm:martingale}.
\begin{proof}[Proof of Theorem~\ref{Thm:martingale}]
  Since \(h^{(\gamma)}\) is non-negative for \(-1\le \gamma\le 1\),
  we have \(M_t^{(\gamma)}\ge 0\).
  Since \((M_t^{(0)},t\ge 0)\) and \((X_t f(L_t),t\ge 0)\) are
  \(((\cF_t),\bP_x)\)-martingales, the process
  \((M_t^{(\gamma)}=M_t^{(0)}+\frac{\gamma}{m^2} X_t f(L_t),t\ge 0)\)
  is also a \(((\cF_t),\bP_x)\)-martingale.
\end{proof}

\section{Local time penalization with two-point hitting time
    clock}\label{Sec:two-hitting-time}
Let us consider the hitting time of
two-point set, i.e.,
\(T_{a,b}=T_{\cbra{a, b}} = T_a \wedge T_b\).
Recall that we assume \(X\) is recurrent
and assume the condition~\ref{item:assumption}.
\subsection{The law of the local time with two-point hitting time clock}
First, we compute \(\bP\sbra{L_{T_a \wedge T_b}},  \bP_x(T_a<T_b\wedge T_c)\)
and \(\bP_x\sbra{f(L_{T_a\wedge T_b})}\) respectively.
\begin{Lem}\label{Lem:LT-two-hitting}
  For \(a\ne b\), it holds that
  \begin{align}
    h^C(a,b)
     & \coloneqq
    \bP\sbra{L_{T_a \wedge T_b}} \\
     & =\frac{1}{h^B(a-b)}
    \cbra*{\begin{multlined}
        \rbra[\big]{h(b)+h(-a)}h(a-b)
        +\rbra[\big]{h(a)+h(-b)}h(b-a)\\
        -h(a-b)h(b-a)
      \end{multlined}}.
  \end{align}
\end{Lem}
\begin{proof}
  For \(q>0\), by the strong Markov property, we have
  \begin{align}\label{eq:L_T_aT_b}
    \bP\sbra*{\int_0^\infty \ce^{-qt}\, \cd L_t}
    =\begin{aligned}[t]
       & \bP\sbra*{\int_0^{T_a\wedge T_b}\ce^{-qt}\, \cd L_t }
      \\
       & + \bP\sbra{\ce^{-qT_a}; T_a<T_b}\bP_a\sbra*{\int_0^\infty \ce^{-qt}\,\cd L_t
      }
      \\
       & + \bP\sbra{\ce^{-qT_b}; T_b<T_a}\bP_b\sbra*{\int_0^\infty \ce^{-qt}\,\cd L_t
      }.
    \end{aligned}
  \end{align}
  Using~\eqref{eq:regularity-of-L} and
  (\ref{Lem-item:two-hittig-time-prob}) of Lemma~\ref{Lem:h^B-hitting-h}
  and letting \(q\to 0+\),
  we obtain the desired result.
\end{proof}

\begin{Lem}\label{Lem:three-hitting-time}
  For \(a,b,c,x\in \bR\) and \(a\ne b, c\), it holds that
  \begin{align}\label{eq:e-three-hitting-time}
   \hspace{-25pt} \bP_x\sbra{\ce^{-qT_a}; T_a<T_b\wedge T_c}
    = \frac{\bP_x\sbra{\ce^{-qT_a}; T_a<T_b}-\bP_x\sbra{\ce^{-qT_c}; T_c< T_b}
      -\bP_c\sbra{\ce^{-qT_a}; T_a< T_b} }
    {1-\bP_a\sbra{\ce^{-qT_c};T_c<T_b}\bP_c\sbra{\ce^{-qT_a}; T_a<T_b}},
  \end{align}
  and letting \(q\to 0+\), it holds that
  \begin{align}\label{eq:three-hitting-time}
    \bP_x(T_a<T_b\wedge T_c)
    = \frac{\bP_x(T_a<T_b)-\bP_x(T_c<T_b)\bP_c(T_a<T_b)}
    {1-\bP_a(T_c<T_b)\bP_c(T_a<T_b)}.
  \end{align}
\end{Lem}
Note that, by~\eqref{eq:two-hittig-time-prob},
we can express
\(\bP_x(T_a<T_b\wedge T_c)\) only in terms of \(h\).
\begin{proof}[Proof of Lemma~\ref{Lem:three-hitting-time}]
  Using the strong Markov property, we have
  \begin{align}\label{eq:pf-three-hitting-times}
    \begin{aligned}
     & \bP_x\sbra{\ce^{-qT_a}; T_a<T_b}    \\
     & =
    \bP_x\sbra{\ce^{-qT_a};T_a<T_b\wedge T_c}
    + \bP_x\sbra{\ce^{-qT_a}; T_c<T_a<T_b} \\
     & =
    \bP_x\sbra{\ce^{-qT_a};T_a<T_b\wedge T_c}
    + \bP_x\sbra{\ce^{-qT_c}; T_c<T_a\wedge T_b}
    \bP_c\sbra{\ce^{-qT_a}; T_a<T_b}.
    \end{aligned}
  \end{align}
  Replacing \(a\) with \(c\), we also have
  \begin{align}\label{eq:pf-three-hitting-times2}
    \begin{aligned}
       & \bP_x\sbra{\ce^{-qT_c}; T_c<T_b} \\
       & =
      \bP_x\sbra{\ce^{-qT_c};T_c<T_a\wedge T_b}
      + \bP_x\sbra{\ce^{-qT_a}; T_a<T_b\wedge T_c}
      \bP_a\sbra{\ce^{-qT_c}; T_c<T_b}.
    \end{aligned}
  \end{align}
  Combining~\eqref{eq:pf-three-hitting-times} and~\eqref{eq:pf-three-hitting-times2},
  we obtain~\eqref{eq:e-three-hitting-time}.
  Letting \(q\to 0+\), we also have~\eqref{eq:three-hitting-time}.
\end{proof}

\begin{Lem}\label{Lem:two-hitting-time-f-L}
  For \(a \ne b\) and \(x\in \bR\), it holds that
  \begin{align}
    \bP_x\sbra{f(L_{T_a\wedge T_b})}
    =
    \bP_x(T_0 > T_a\wedge T_b) f(0)
    + \frac{\bP_x(T_0 < T_a\wedge T_b)}{h^C(a,b)}
    \int_0^\infty \ce^{-u/h^C(a,b)}f(u)
    \, \cd u.
  \end{align}
\end{Lem}
\begin{proof}
  In the same way as~\eqref{eq:L_T_a-exp-distrib},
  we have, for \(l>0\),
  \begin{align}
    \bP(L_{T_a\wedge T_b}>l)=
    \bP(\sigma_{\cbra{T_a\wedge T_b <T_0}}>l)
    = \ce^{-tn(T_a\wedge T_b <T_0)}.
  \end{align}
  In particular,
  \(L_{T_a\wedge T_b}\) is exponentially distributed.
  (Consequently, we obtain \(n(T_a\wedge T_b<T_0)=1/\bP\sbra{L_{T_a,b}}=1/h^C(a,b)\).)
  We may apply the strong Markov property and obtain the desired result.
\end{proof}

\subsection{Proof of
  Theorem~\ref{Thm:two-point-hitting-time-result}}\label{Subsec:pf-two-hitting}
We are now ready to prove the two-point hitting time result.
Recall that \((a,b)\xrightarrow[]{\gamma}\infty\) means~\eqref{eq:a-b-inf-notation}.

\begin{proof}[Proof of Theorem~\ref{Thm:two-point-hitting-time-result}]
  By Lemmas~\ref{Lem:LT-two-hitting} and~\ref{Lem:three-hitting-time}
  and~\eqref{eq:two-hittig-time-prob}, we have
  \begin{align}
     & h^C(a,b)\bP_x(T_0>T_a\wedge T_b)
    \\
     & =h(x) + \frac{1}{h^B(a-b)}
    \cbra*{\begin{multlined}
        \rbra[\big]{h(-a)-h(x-a)}h(a-b)
        + \rbra[\big]{h(-b)-h(x-b)}h(b-a)\\
        -\rbra[\big]{h(a)-h(b)}\rbra[\big]{h(-a)-h(x-a)-h(-b)+h(x-b)}
      \end{multlined}}.
  \end{align}
  Recall that \(h^B(a)=h(a)+h(-a)\); see~\eqref{eq:h^B}.
  Replacing \(b\) with \(-b\) and
  using Theorem~\ref{Thm:property-h}, it holds that
  \begin{align}
    h^C(a, -b) \bP_x(T_0>T_{a}\wedge T_{-b})
    \xrightarrow[(a,b)\xrightarrow{\gamma}\infty]{}
    h^{(\gamma)}(x).
  \end{align}
  By the strong Markov property and Lemma~\ref{Lem:two-hitting-time-f-L},
  we have
  \begin{align}
    N_t^{a,b} & = 1_{\cbra{t<T_{a,-b}}}
    \cbra*{h^C(a,-b)\bP_{X_t}(T_0>T_{a,-b})f(L_t)} \\
              & \quad +1_{\cbra{t<T_{a,-b}}}
    \cbra*{\bP_{X_t}(T_0<T_{a,-b})\int_0^\infty \ce^{-u/h^C(a,-b)} f(L_t+u)\, \cd y}.
  \end{align}
  Hence we obtain
  \(N_t^{a,b}\to M_t^{(\gamma)}\),
  \(\bP_x\)-a.s.\ as \((a,b)\xrightarrow[]{\gamma}\infty\).
  The proof of \(M_t^{a,b}\to M_t^{(\gamma)}\),
  \(\bP_x\)-a.s.\
  is similar to~\eqref{eq:hitting-A}--\eqref{eq:hitting-Al}.
  Since \((M_t^{(\gamma)},t\ge 0)\) is
  a non-negative \(((\cF_t),\bP_x)\)-martingale (Theorem~\ref{Thm:martingale}),
  we may apply Scheff\'{e}'s lemma to obtain
  the \(\cL^1\) convergence.
\end{proof}

\section{Local time penalization with inverse local time clock}\label{Sec:inv-time}
We define the modified Bessel function of the first kind, which
is expressed as
\begin{align}
  I_{\nu}(x) = \sum_{n=0}^\infty \frac{{(x/2)}^{\nu+2n}}{n! \Gamma(\nu+n+1)},
  \quad \nu \ge 0, \,\, x > 0.
\end{align}
Note that \(I_{\nu}(x)\) is increasing in \(x>0\).
For more details, see e.g.,~\cite[Section 5]{MR0350075}.
Recall that \(\eta_u^a\) denotes the inverse local time at \(a\):
\(
\eta_u^a = \inf\cbra{t\ge 0\colon L_t^a > u}
\).
We consider the penalization with inverse local time clock
in two ways:
first, we make
\(a\) tend to infinity, and second, \(u\) tend to infinity.
Recall that we assume \(X\) is recurrent
and assume the condition~\ref{item:assumption}.

\subsection{The law of the local time with inverse local time clock}
\begin{Lem}
  Let \(a \in \bR\setminus\cbra{0}\).
  Then the process \((L_{\eta^a_u},u\ge 0)\) under \(\bP_a\)
  is a compound Poisson process with Laplace transform
  \begin{align}\label{eq:L_eta^a-laplace}
    \bP_a\sbra{\ce^{-\beta L_{\eta^a_u}}}
    = \ce^{-u\beta/(1+\beta h^B(a))}, \quad \beta\ge 0.
  \end{align}
  Moreover, for any \(u>0\) and \(f\in\cL^1_{+}\),
  it holds that
  \begin{align}
    \bP_a\sbra{f(L_{\eta^a_u})}
    = \ce^{-u/h^B(a)} f(0)
    + \int_0^\infty f(y)\rho^{h^B(a)}_u(y)
    \, \cd y,
  \end{align}
  where
  \begin{align}
    \rho^a_u(y) = \ce^{-(u+y)/a}\frac{\sqrt{u/y}}{a}
    I_1\rbra*{\frac{2\sqrt{uy}}{a}}.
  \end{align}
\end{Lem}
We omit the proof because it is very similar to that in the
diffusion case of Profeta--Yano--Yano~\cite[Lemma 4.1]{MR3909919}, using
Theorem~\ref{Thm:n-two-hitting-time}.
For the proof, we use
\(n^a(T_0<T_a) = n(T_{-a}<T_0) = 1/h^B(a)\)
in the L\'{e}vy case, instead of
\(n^a(T_0<T_a) = n(T_a<T_0) = 1/a\)
in the diffusion case.
The proof of~\eqref{eq:L_eta^a-laplace} can also be found
in~\cite[Lemma V.13]{MR1406564}.

The proof of the next lemma is completely parallel
to that of~\cite[Lemma 4.2]{MR3909919}.
So we omit it.
\begin{Lem}
  For \(u > 0, \; x,a \in\bR\), it holds that
  \begin{align}
    \bP_x\sbra{f(L_{\eta^a_u})}
     & = \bP_x(T_a<T_0)\bP_a\sbra{f(L_{\eta^a_u})}
    + \bP_x(T_0<T_a)\bP_a\sbra{f(\bm{e}_{1/h^B(a)}+L_{\eta^a_u})} \\
     & = \bP_x(T_a<T_0)\bP_a\sbra{f(L_{\eta^a_u})}
    + \frac{\bP_x(T_0<T_a)}{h^B(a)}
    \int_0^\infty f(y)\widetilde{\rho}^{h^B(a)}_u(y)
    \, \cd y,
  \end{align}
  where
  \begin{align}
    \widetilde{\rho}^a_u(y)
    = \ce^{-(u+y)/a}I_0\rbra*{\frac{2\sqrt{uy}}{a}}.
  \end{align}
\end{Lem}

\subsection{Limit as \(a\) tends to infinity with \(u\) being fixed}
\begin{Thm}\label{Thm:inv-LT-result}
  Let \(f\in \cL^1_{+}\) and \(x\in\bR\).
  For any \(u> 0\) and \(a\in \bR\), we define
  \begin{align}
    N_t^{a,u}
     & = h^B(a)\bP_x\sbra{f(L_{\eta^a_u}); t<\eta^a_u|\cF_t}, \\
    M_t^{a,u}
     & = h^B(a)\bP_x\sbra{f(L_{\eta^a_u})|\cF_t}.
  \end{align}
  Then
  \begin{align}
    \lim_{a\to\pm\infty}N_t^{a,u}
    = \lim_{a\to\pm\infty}M_t^{a,u}
    = M_t^{(\pm 1)},\quad
    \text{\(\bP_x\)-a.s.\ and in \(\cL^1(\bP_x)\).}
  \end{align}
  Consequently, if \(M_0^{(\pm 1)}>0\) under \(\bP_x\),
  it holds that
  \begin{align}
    \frac{\bP_x\sbra{F_t f(L_{\eta^a_u})}}{\bP_x\sbra{f(L_{\eta^a_u})}}
    \longrightarrow \bP_x\sbra*{F_t \frac{M_t^{(\pm 1)}}{M_0^{(\pm 1)}}},
    \qquad \text{as \(a \to \pm\infty \),}
  \end{align}
  for all bounded \(\cF_t\)-measurable functionals \(F_t\).
\end{Thm}
The proof of the theorem is
very similar to that of~\cite[Lemma 4.4 and Theorem 4.5]{MR3909919}.
So we omit it.
(\cite[Lemma 4.4]{MR3909919} states only convergence
in probability but, its a.s.\ convergence
can also be proved by the same proof.)

\subsection{Limit as \(u\) tends to infinity with \(a\) being fixed}
In this section, we only consider the cases \(f(x) = \ce^{-\beta x}\)
and \(f(x) = 1_{\cbra{x=0}}\).
The next theorem is in the case \(f(x)= \ce^{-\beta x}\).
\begin{Thm}\label{Thm:inv-u-e}
  Let \(x \in \bR\), \(a \in \bR\setminus\cbra{0}\), \(\beta>0\) and \(t>0\).
  Define
  \begin{align}
    N_t^{u,\beta, a} & =
    \ce^{\beta u/(1+\beta h^B(a))}
    \bP_x\sbra{\ce^{-\beta L_{\eta^a_u}}; t<\eta^a_u | \cF_t}, \\
    M_t^{u,\beta, a} & =
    \ce^{\beta u/(1+\beta h^B(a))}
    \bP_x\sbra{\ce^{-\beta L_{\eta^a_u}} | \cF_t},
  \end{align}
  and
  \begin{align}
    M_t^{\beta,a}
    = \ce^{-\beta L_t}
    \cbra*{\bP_{X_t}(T_a<T_0)
      + \frac{\bP_{X_t}(T_0<T_a)}{1+\beta h^B(a)}\ce^{\beta L_t^a/(1+\beta h^B(a))}}.
  \end{align}
  Then it holds that
  \begin{align}
    \lim_{u\to\infty}N_t^{u,\beta,a}
    =\lim_{u\to\infty}M_t^{u,\beta,a}
    = M_t^{\beta,a},\quad
    \text{\(\bP_x\)-a.s.\
      and in \(\cL^1(\bP_x)\).}
  \end{align}
\end{Thm}
The proof of Theorem~\ref{Thm:inv-u-e} is parallel
to that of~\cite[Lemma 4.6 and Theorem 4.7]{MR3909919}.
So we omit it.

We consider the case \(f(x) = 1_{\cbra{x=0}}\).
\begin{Thm}\label{Thm:inv-u-1}
  Let \(x \in \bR\), \(a \in \bR\setminus\cbra{0}\), \(\beta>0\) and \(t>0\).
  Define
  \begin{align}
    N_t^{u,\infty,a}
     & = \ce^{u/h^B(a)}
    \bP_x\rbra{t<\eta_u^a<T_0|\cF_t}, \\
    M_t^{u,\infty,a}
     & = \ce^{u/h^B(a)}
    \bP_x\rbra{\eta_u^a<T_0|\cF_t},
  \end{align}
  and
  \begin{align}
    M_t^{\infty,a}
    = \ce^{L_t^a/h^B(a)} \bP_{X_t}(T_a<T_0) 1_{\cbra{t<T_0}}.
  \end{align}
  Then it holds that
  \begin{align}
    \lim_{u\to\infty}N_t^{u,\infty,a}
    =\lim_{u\to\infty}M_t^{u,\infty,a}
    =M_t^{\infty,a},\quad
    \text{\(\bP_x\)-a.s.\
      and in \(\cL^1(\bP_x)\).}
  \end{align}
\end{Thm}
The proof of Theorem~\ref{Thm:inv-u-1} is very similar to
that of~\cite[Theorem 4.8]{MR3909919}.
So we omit it.

\section{Universal \(\sigma\)-finite measures}\label{Sec:universal-measure}
In this section, we shall discuss the penalized processes and
\(\sigma\)-finite measures unifying the processes.
We have obtained the penalized measure \(\bQ^{(\gamma,f)}_x\),
which is given by
\begin{align}
  \bQ^{(\gamma,f)}_x|_{\cF_t}
  = \frac{M_t^{(\gamma,f)}}{M_0^{(\gamma,f)}}\cdot\bP_x|_{\cF_t},
  \quad -1\le \gamma \le 1.
\end{align}

\subsection{L\'{e}vy processes conditioned to avoid zero}
We show that \(h^{(\gamma)}\) is invariant for the killed process.
\begin{Thm}\label{Thm:hg-invariant}
  For \(-1\le \gamma\le 1\), it holds that
  \begin{align}
    \bP_x\sbra{h^{(\gamma)}(X_t);T_0>t} = h^{(\gamma)}(x), \quad
    n\sbra{h^{(\gamma)}(X_t);T_0>t} = 1, \quad
    x \in \bR.
  \end{align}
\end{Thm}
\begin{proof}
  We can show the case \(\gamma=0\)
  by the completely same discussion as the proof
  of Pant\'{\i}~\cite[(iii) of Theorem 2.2]{MR3689384}.
  Combining this with Theorem~\ref{Thm:mart-XfL}, we obtain the desired result.
  The former equation also follows from the fact that
  \((M_t^{(\gamma, 1_{\cbra{u=0}})},t\ge 0)\) is a \(((\cF_t),\bP_x)\)-martingale.
\end{proof}

Let \(\cH^{(\gamma)} = \cbra{x\in\bR\colon h^{(\gamma)}(x)>0 }\) and
\(\cH^{(\gamma)}_0 = \cH^{(\gamma)} \cup \cbra{0}\).
We introduce the \(h^{(\gamma)}\) transformed process given by
\begin{align}\label{eq:h-g-trans}
  \bP_x^{(\gamma)}|_{\cF_t}
  = \begin{dcases}
    1_{\cbra{T_0>t}} \frac{h^{(\gamma)}(X_t)}{h^{(\gamma)}(x)}
    \cdot \bP_x|_{\cF_t}               & \text{if } x\in\cH^{(\gamma)}, \\
    1_{\cbra{T_0>t}}  h^{(\gamma)}(X_t) \cdot n|_{\cF_t} & \text{if } x = 0.
  \end{dcases}
\end{align}
Since \(\bP_x^{(\gamma)}|_{\cF_t}\) is consistent in \(t>0\),
the probability measure \(\bP_x^{(\gamma)}\) can be well-defined
on \(\cF_\infty\coloneqq \sigma(X_t,t\ge 0)\),
for more details, see Yano~\cite[Theorem 9.1]{yano2021universality}.
For any \(t>0\), we have
\(\bP_x^{(\gamma)}\rbra{T_{\bR \setminus \cH^{(\gamma)} } > t} = 1\).
Consequently, we have
\(\bP_x^{(\gamma)}\rbra{T_{\bR \setminus\cH^{(\gamma)}} = \infty} = 1\) and
in particular, \(\bP_x^{(\gamma)}\rbra{T_0= \infty}=1\).
The process \(\bP_x^{(\gamma)}\)
is called a \textit{L\'{e}vy process conditioned to avoid zero}.
Note that, for \(x\in\cH^{(\gamma)}\),
the measure \(\bP_x^{(\gamma)}\) is absolutely continuous with respect to
\(\bP_x\) on \(\cF_t\), but is
singular to \(\bP_x\) on \(\cF_\infty\) since \(\bP_x(T_0<\infty)=1\).

By Theorems~\ref{Thm:exp-time-result},~\ref{Thm:hitting-time-result},~\ref{Thm:two-point-hitting-time-result}
and~\ref{Thm:inv-LT-result}
(Corollaries~\ref{Cor:hitting-cond} and~\ref{Cor:two-point-hitting-cond})
and by taking \(f=1_{\cbra{u=0}}\),
we have the following conditioning results.
\begin{Cor}\label{Cor:avoid-zero}
  Let \(t>0\) and \(F_t\) be a bounded \(\cF_t\)-measurable functional.
  Then the following assertions hold:
  \begin{enumerate}
    \item
    \(\displaystyle
    \lim_{q\to 0+}
    \bP_x\sbra{F_t|T_0>\bm{e}_q}=\bP_x^{(0)}\sbra{F_t}
    \), for \(x\in \cH^{(0)}\);\label{Cor-item:avoid-zero-exp}
    \item
    \(\displaystyle
      \lim_{a\to\pm\infty}
      \bP_x\sbra{F_t|T_0>T_a}=\bP_x^{(\pm 1)}\sbra{F_t}
    \), for \(x\in \cH^{(\pm 1)}\);
    \item
    \(\displaystyle
      \lim_{(a,b)\xrightarrow[]{\gamma}\infty}
      \bP_x\sbra{F_t|T_0>T_{a,-b}}=\bP_x^{(\gamma)}\sbra{F_t}
    \), for \(-1\le \gamma\le 1\) and \(x\in \cH^{(\gamma)}\);
    \item
    \(\displaystyle
      \lim_{a\to\pm\infty}
      \bP_x\sbra{F_t|T_0>\eta^a_u}=\bP_x^{(\pm 1)}\sbra{F_t}
    \), for \(u>0\) and \(x\in \cH^{(\pm 1)}\).
  \end{enumerate}
\end{Cor}
Note that (\ref{Cor-item:avoid-zero-exp}) of Corollary~\ref{Cor:avoid-zero}
generalizes Pant\'{\i}~\cite[Theorem 2.7]{MR3689384}.

\subsection{Universal \(\sigma\)-finite measures}
In this subsection, we assume that \((X,\bP_x)\) has
a transition density \(p_t(\cdot)\).
Then we can construct the L\'{e}vy bridge.
Let \(\bP_{x,y}^u\) denote the law of bridge
from \(X_0=x\) to \(X_u=y\).
This measure can be constructed as
\begin{align}
  \bP_{x,y}^u(A) = \bP_x\sbra*{1_A\frac{p_{u-t}(y-X_t)}{p_t(y-x)}},
  \quad A\in\cF_t,\; 0< t<u.
\end{align}
See Fitzsimmons--Pitman--Yor~\cite{MR1278079}.
We have the conditioning formula:
\begin{align}
  \bP_x\sbra*{\int_0^t F_u \, \cd L_u}
  = \int_0^t \bP_x\sbra{\cd L_u}\bP_{x, 0}^{u}\sbra{F_u},
  \quad t>0,
\end{align}
for all non-negative predictable processes \((F_u)\), where
we write symbolically
\( \bP_x\sbra{\cd L_u} = p_u(-x) \, \cd u\).

For \(x\in \bR\) and \(-1\le \gamma\le 1\), we define
\begin{align}
  \cP_x^{(\gamma)}
   & = \int_0^\infty \bP_x\sbra{\cd L_u}
  \rbra*{\bP_{x,0}^{u} \bullet \bP_0^{(\gamma)}}
  + h^{(\gamma)}(x)\bP_x^{(\gamma)},
\end{align}
where
the symbol \(\bullet\) stands for the concatenation and
\(h^{(\gamma)}(x)\bP_x^{(\gamma)} = 0\)
for \(x \in \bR\setminus\cH^{(\gamma)}\).
Then we have the following:
\begin{Thm}
  Let \(x \in \bR\) and \(f\in \cL^1_{+}\).
  Let \(t>0\) and \(F_t\) be a bounded \(\cF_t\)-measurable
  functional.
  Then the following assertions hold:
  \begin{enumerate}
    \item
          \(\displaystyle
          \lim_{q\to 0+}
          r_q(0) \bP_x\sbra{F_t f(L_{\bm{e}_q})}
          = \cP_x^{(0)} \sbra{F_t f(L_{\infty})}\);
    \item
          \(\displaystyle
          \lim_{a\to\pm\infty}
          h^B(a)\bP_x\sbra{F_t f(L_{T_a})}
          = \cP_x^{(\pm 1)} \sbra{F_t f(L_{\infty})}\);
    \item
          \(\displaystyle
          \lim_{(a,b)\xrightarrow[]{\gamma}\infty}
          h^C(a, -b)
          \bP_x\sbra{F_t f(L_{T_{a, -b}})}
          =\cP_x^{(\gamma)} \sbra{F_t f(L_{\infty})}\),
          for \(-1\le\gamma\le 1\);
    \item
          \(\displaystyle
          \lim_{a\to\pm\infty}
          h^B(a)\bP_x\sbra{F_t f(L_{\eta^a_u})}
          = \cP_x^{(\pm 1)} \sbra{F_t f(L_{\infty})}\), for \(u>0\).
  \end{enumerate}
\end{Thm}
\begin{proof}
  It suffices to show that
  \begin{align}
    \cP^{(\gamma)}_x\sbra{F_t f(L_\infty)} = \bP_x\sbra{F_t M_t^{(\gamma)}},
  \end{align}
  for \(-1 \le \gamma \le 1\).
  The proof is the same as that of Theorem 5.3 of~\cite{MR3909919}.
\end{proof}
Consequently, we obtain the representation of \(\bQ^{(\gamma,f)}_x\) as follows:
\begin{align}
  \bQ^{(\gamma,f)}_x
  = \frac{f(L_\infty)}{\cP^{(\gamma)}_x\sbra{f(L_\infty)}}
  \cdot \cP^{(\gamma)}_x.
\end{align}
Recall that \(g = \sup\{t\colon X_t=0\}\).
Since \(\bO^{(\gamma,f)}_x(g < \infty) = \bP_x(g = \infty)= 1\),
the two measures are singular on \(\cF_\infty\).

\subsection{The law of \(L_\infty\) under \(\bQ_0^{(\gamma,f)}\)}
We assume,
for simplicity, that \(f\in \cL^1_{+}\)
satisfies \(\int_0^\infty f(u)\, \cd u = 1\).
Then we have \(M_0^{(\gamma,f)}=1\), \(\bP_0\)-a.s.
For \(l\ge 0\), the optional stopping theorem implies that
\begin{align}
  \bQ_0^{(\gamma,f)}(L_t\ge l)
  = \bP_0\sbra{M_t^{(\gamma, f)}; L_t\ge l}
  = \bP_0\sbra{M_{\eta_l}^{(\gamma, f)}; \eta_l\le t}.
\end{align}
Letting \(t\to\infty\), we may apply the monotone convergence theorem to
deduce that
\begin{align}
  \bQ_0^{(\gamma,f)}(L_\infty\ge l) = \bP_0\sbra{M_{\eta_l}^{(\gamma, f)}}.
\end{align}
Since \(X_{\eta_l}= 0\) and \(L_{\eta_l}=l\), we have
\begin{align}
  \bP_0\sbra{M_{\eta_l}^{(\gamma, f)}}
  = \int_0^\infty f(l+u)\, du.
\end{align}
Therefore, it holds that
\begin{align}
  \bQ_0^{(\gamma,f)}(L_\infty \in \cd u) = f(u)\, \cd u,
  \quad u>0.
\end{align}

\section{The transient case}\label{Sec:trans}
We now study penalization in the transient case.
Throughout this section, we always assume
\(X\) is transient and
the conditions~\ref{item:cond-not-CPP} and~\ref{item:cond-regular} hold.
Recall that
\begin{align}
  \kappa\coloneqq \lim_{q\to 0+}\frac{1}{r_q(0)}
  = n(T_0=\infty)>0;
\end{align}
see~\eqref{eq:kappa}.
\subsection{The renormalized zero resolvent in the transient case}
As in the recurrent case, we define \(h_q(x)=r_q(0)-r_q(-x)\).
\begin{Thm}
  Suppose that
  the conditions~\ref{item:cond-not-CPP} and~\ref{item:cond-regular} hold.
  Then the following assertions hold.
  \begin{enumerate}
    \item For any \(x\in\bR\), it holds that
          \(\displaystyle
          h(x)\coloneqq \lim_{q\to 0+} h_q(x) = \kappa^{-1}\bP_x(T_0=\infty).
          \)
    \item The above convergence is uniform on compacts,
          and consequently \(h\) is continuous.\label{Thm-item:h-unif-conv-trans}
    \item \(h\) is subadditive on \(\bR\),
          that is, \(h(x+y)\le h(x)+h(y)\) for
          \(x,y\in\bR\).\label{Thm-item:h-subadd-trans}
  \end{enumerate}
\end{Thm}

\begin{proof}
  It follows from~\eqref{eq:exp-hit-time}
  and~\eqref{eq:kappa} that
  \begin{align}\label{eq:transient-h-conv}
    h_q(x) = r_q(0)\bP_x\sbra{1-\ce^{-qT_0}}
    \longrightarrow \kappa^{-1} \bP_x(T_0 = \infty),
    \quad \text{as \(q \to 0+\).}
  \end{align}
  The proof of (\ref{Thm-item:h-unif-conv-trans}) and (\ref{Thm-item:h-subadd-trans})
  are the same as that of Theorem~\ref{Thm:exist-h}.
\end{proof}

\begin{Thm}
  Suppose that
  the conditions~\ref{item:cond-not-CPP} and~\ref{item:cond-regular} hold.
  Then the following assertions hold:
  \begin{enumerate}
    \item
          \(\displaystyle
          \lim_{x\to\pm\infty} \frac{h(x)}{\abs{x}} =0
          \);\label{Thm-item:h/x-infinity-trans}
    \item
          \(\displaystyle
          \lim_{y\to\pm\infty} \cbra*{h(x+y) - h(y)} =0,
          \) for all \(x\in\bR\).\label{Thm-item:h-h-limit-trans}
  \end{enumerate}
\end{Thm}
\begin{proof}
  Since \(h(x)\le \kappa^{-1}\), it is
  obvious that~\eqref{Thm-item:h/x-infinity-trans} holds.
  The proof of~\eqref{Thm-item:h-h-limit-trans} is the same as that of
  (\ref{Thm-item:h-h-limit}) of Theorem~\ref{Thm:property-h}
  in the recurrent and \(m^2=\infty\) case.
\end{proof}
\subsection{Useful equations}
Before stating out penalization result, we introduce
some useful equations.
\begin{Lem}\label{Lem:useful-eq-trans}
  \begin{enumerate}
    \item For \(a\in\bR\),
          \begin{align}
            h^B(a) \coloneqq \lim_{q\to 0+} h_q^B(a)
            =\bP_0\sbra{L_{T_a}}= h(a) + h(-a) - \kappa h(a)h(-a).
            \label{eq:h^B-trans}
          \end{align}\label{Lem-item:hq-two-hitting-time-trans}
    \item For \(x,a,b\in\bR\) and \(a\ne b\),
          \begin{align}
            \bP_x(T_a < T_b)
            = \frac{h(b-a)+h(x-b)-h(x-a) - \kappa h(x-b)h(b-a)}{h^B(a-b)}.
            \label{eq:two-hittig-time-prob-trans}
          \end{align}
    \item
          For \(x,a,b\in\bR\) and \(a\ne b\),
          \begin{align}\label{eq:L_T_aT_b-trans}
            \bP\sbra{L_{T_a \wedge T_b}}
             & =\frac{1}{h^B(a-b)}
            \cbra*{\begin{multlined}
                \rbra[\big]{h(b)+h(-a)-\kappa h(-a)h(b)}h(a-b)	\\
                \qquad +\rbra[\big]{h(a)+h(-b)-\kappa h(-b)h(a)}h(b-a) \\
                -h(a-b)h(b-a)
              \end{multlined}}.
          \end{align}
  \end{enumerate}
\end{Lem}
The proof of Lemma~\ref{Lem:useful-eq-trans} is similar to
that in the recurrent case.
So we omit it.

\begin{Lem}[{\cite[Lemma 3.10]{MR3689384}}]\label{Thm:h^B-infty-trans}
  It holds that \(\lim_{x\to\infty} h^B(x) = \kappa^{-1}\).
\end{Lem}
\begin{proof}
  For completeness of the paper, we give the same proof as
  that of~\cite[Lemma 3.10]{MR3689384}.
  Let \(\bm{e}\) be the exponentially distributed with parameter
  \(1\) and \(\bm{e}_q \coloneqq \bm{e}/q\) for \(q>0\).
  Then we already have~\eqref{eq:h_b-eq} and~\eqref{eq:h_b-lower}.
  Letting \(q\to 0+\) in~\eqref{eq:h_b-lower}, we obtain
  \begin{align}
    \liminf_{x\to\infty}h^B(x) \ge \kappa^{-1}.
  \end{align}
  On the other hand, it holds that
  \begin{align}
    h^B(x) \le \bP\sbra{L_{\bm{e}_q}} + \bP\sbra{L_{T_x}; T_x>\bm{e}_q}
    = r_q(0) + \bP\sbra{L_{T_x}; T_x>\bm{e}_q}.
  \end{align}
  Letting \(q\to 0+\), we have
  \(h^B(x) \le \kappa^{-1}\).
  Therefore we obtain the desired result.
\end{proof}

\begin{Thm}\label{Thm:n-two-hitting-time-trans}
  For \(a \in \bR \setminus \{0\}\), it holds that
  \begin{align}
    n(T_a < T_0 < \infty) = \frac{1-\kappa h^B(a)}{h^B(a)}, \quad
    \text{and} \quad
    n(T_a < T_0) = \frac{1-\kappa h(-a)}{h^B(a)}.
  \end{align}
\end{Thm}
\begin{proof}
  For \(l > 0\), it holds that
  \begin{align}
    \bP(L_{T_a}>l) = \bP(T_a > \eta_l) =
    \bP(\sigma_{\cbra{T_0 =\infty}\cup \cbra{T_a<T_0<\infty}}> l),
  \end{align}
  where \(\sigma_A = \inf\cbra{l \colon e_l \in A}\)
  for \(A \subset \cD\).
  Since \(\sigma_A\) is the hitting time
  of the set \(A\)
  for the killed
  Poisson point process \(((l, e_l), l\ge 0)\),
  we have
  \begin{align}\label{eq:L_T_a-exp-distrib-trans}
    \bP(L_{T_a}>l)
    = \ce^{-l\cbra{n(T_0=\infty)+n(T_a<T_0<\infty)}}
    = \ce^{-l\cbra{\kappa +n(T_a<T_0<\infty)}}.
  \end{align}
  For more details, see, e.g.,~\cite[Lemma 6.17]{MR3155252}.
  In particular, \(L_{T_a}\) is exponentially distributed.
  On the other hand, we know \(h^B(a) = \bP\sbra{L_{T_a}}\)
  by Lemma~\ref{Lem:h^B-hitting-h}.
  Hence we obtain
  \begin{align}
    n(T_a < T_0 < \infty) = \frac{1}{h^B(a)} - \kappa.
    \label{eq:n-T_a-T_0-inf-trans}
  \end{align}
  We use the strong Markov property of the excursion measure \(n\)
  (see, e.g.,~\cite[Theorem III.3.28]{MR1138461})
  to obtain
  \begin{align}
    n(T_a < T_0 < \infty)
    = n(\bP_a(T_0 <\infty); T_a<T_0)
    = \rbra{1-\kappa h(a)} n(T_a<T_0).
    \label{eq:n-T_a-T_0}
  \end{align}
  It follows from~\eqref{eq:n-T_a-T_0-inf} and~\eqref{eq:n-T_a-T_0}
  that
  \begin{align}
    n(T_a<T_0)
    = \frac{1-\kappa h^B(a)}{h^B(a)(1-\kappa h(a))}
    = \frac{1-\kappa h(-a)}{h^B(a)}.
  \end{align}
  Hence the proof is complete.
\end{proof}

\subsection{Penalization result in the transient case}
Let \(f\) be a non-negative
function on \([0,\infty)\)
which satisfies
\(\int_0^\infty \ce^{-\kappa u}f(u)\, \cd u <\infty\).
For \(x \in \mathbb{R}\), we introduce the process given by
\begin{align}\label{eq:def-mart-trans}
  M_t= M_t^{(f)}= h(X_t)f(L_t)+(1-\kappa h(X_t))
  \int_0^\infty
  \ce^{-\kappa u} f(L_t+u)\,\cd u.
\end{align}
Then we can show that \((M_t,t\ge 0)\) is a non-negative martingale.
\begin{Thm}\label{Thm:transient-infty-mart}
  Let \(x \in \mathbb{R}\) and let \(f\) be a non-negative
  function on \([0,\infty)\)
  which satisfies
  \(\int_0^\infty \ce^{-\kappa u}f(u)\, \cd u <\infty\).
  Then it holds that
  \begin{align}\label{eq:L-inf-M-trans}
    \bP_x\sbra{f(L_\infty)}
    = \kappa M_0,\quad \bP_x\text{-a.s.,}
  \end{align}
  and \((M_t,t\ge 0)\) is a non-negative \(((\cF_t),\bP_x)\)-martingale.
\end{Thm}
\begin{proof}
  By the same discussion, Lemma~\ref{Lem:expect-L_e_q} also holds
  in the transient case.
  Suppose first that \(f\) is bounded.
  We write \(g = \sup\{t\colon X_t=0\}\).
  Since \(X\) is transient, \(\bP_x(g < \infty)=1\).
  We see that \(f(L_{\bm{e}_q})\to f(L_\infty)\), \(\bP_x\)-a.s.\ as \(q\to 0+\); in fact,
  for almost every sample path,
 \(L_{\bm{e}_q} = L_g = L_\infty\) for small \(q>0\)
 (Here we do not need continuity of \(f\)).
  Hence, by the dominated convergence theorem, we obtain
  \begin{align}
    \bP_x\sbra{f(L_{\bm{e}_q})} \longrightarrow \bP_x\sbra{f(L_\infty)},
    \qquad \text{as \(q \to 0+\).}
  \end{align}
  On the other hand, by the monotone convergence theorem, we obtain
  \begin{align}
    \int_0^\infty \ce^{-u/{r_q(0)}}f(u) \, \cd u
    \longrightarrow \int_0^\infty \ce^{-\kappa u} f(u)\, \cd u,
    \qquad  \text{as } q \to {0+}.
  \end{align}
  Hence~\eqref{eq:L-inf-M-trans} follows by letting \(q \to 0+\)
  in~\eqref{eq:lem-expect-L_e_q}.
  To remove the boundedness assumption of \(f\),
  we consider \(f \wedge n\) and then let \(n \to \infty\) in~\eqref{eq:L-inf-M-trans}.
  Moreover, by the Markov property and the additivity of \(L\), we have,
  for \(0<s<t\),
  \begin{align}
    \bP_x\sbra{M_t|\cF_s}
     & =\kappa^{-1}
    \bP_x\sbra[\big]{\widetilde{\bP}_{X_t}\sbra{f(\widetilde{L}_\infty+L_t)}|\cF_s}
    \\
     & =\kappa^{-1}\bP_x\sbra[\big]{{\bP}_{x}\sbra{f({L}_\infty)|\cF_t}|\cF_s}
    \\
     & = \kappa^{-1}{\bP}_{x}\sbra{f({L}_\infty)|\cF_s}
    \\
     & = \kappa^{-1}\widetilde{\bP}_{X_s}\sbra{f(\widetilde{L}_\infty+L_s)}    \\
     & = M_s.
  \end{align}
  This means that \((M_t,t\ge 0)\) is a non-negative \(((\cF_t),\bP_x)\)-martingale.
\end{proof}

\begin{Thm}\label{Thm:penal-trans}
  Suppose that
  the conditions~\ref{item:cond-not-CPP} and~\ref{item:cond-regular} hold.
  Let \(f\) be a bounded non-negative function
  and let \(\tau\) be a random clock. Define
  \begin{align}
    M_t^\tau & = \kappa^{-1}\bP_x\sbra{ f(L_\tau)|\cF_t}.
  \end{align}
  Then it holds that
  \begin{align}
    M_t^\tau \longrightarrow M_t ,
    \quad \text{\(\bP_x\)-a.s.\ and in \(\cL^1(\bP_x)\) as \(\tau\to\infty\).}
  \end{align}
  Consequently, if \(M_0>0\) under \(\bP_x\),
  it holds that
  \begin{align}\label{eq:penalized-meas-trans}
    \frac{\bP_x\sbra{F_t f(L_{\tau})}}{\bP_x\sbra{f(L_{\tau})}}
    \longrightarrow \bP_x\sbra*{F_t \frac{M_t}{M_0}},
    \qquad \text{as \(\tau \to \infty\),}
  \end{align}
  for all bounded \(\cF_t\)-measurable functionals \(F_t\).
\end{Thm}
\begin{proof}
  We have
  \(f(L_\tau)\to f(L_\infty)\), \(\bP_x\)-a.s.
  as \(\tau\to\infty\); in fact,
  \(L_\tau=L_g=L_\infty\) for large \(\tau\).
  In addition, since \(f\) is bounded,
  we may apply the dominated convergence theorem to obtain
  \(f(L_\tau)\to f(L_\infty)\), in
  \(\cL^1(\bP_x)\) as \(\tau\to\infty\).
  This implies that \(M_t^\tau \to M_t\), \(\bP_x\)-a.s.\ and in
  \(\cL^1(\bP_x)\) as \(\tau\to\infty\).
\end{proof}
\begin{Rem}
  If \(\tau\) is exponential clock,
  hitting clock, two-point hitting time clock or
  inverse local time clock, Theorem~\ref{Thm:penal-trans}
  also holds under the assumption that
  \(f\) is a non-negative
  function which satisfies
  \(\int_0^\infty \ce^{-\kappa u}f(u)\, \cd u <\infty\).
\end{Rem}

Since we have \(M_t = \bP_x\sbra{f(L_\infty)|\cF_t}\),
we see that the penalized measure \(\bQ^f_x\) can be represented as
\begin{align}
  \bQ^f_x = \frac{f(L_\infty)}{\bP_x\sbra{f(L_\infty)}} \cdot \bP_x,
\end{align}
which shows that \(\bQ^f_x\) is absolutely continuous with respect to \(\bP_x\).

\section{Appendix: Martingale property of \(X_t f(L_t)\)}\label{Sec:mart}
In Remark~\ref{Rem:martingale}, we have shown that
\((X_t f(L_t),t\ge 0)\) is a \(((\cF_t),\bP_x)\)-martingale
for \(f\in\cL^1_{+}\) under the condition that \(m^2<\infty\).
Let us remove the additional assumption \(m^2<\infty\).
In this section, we assume \(X\) is either recurrent or transient,
and assume the conditions~\ref{item:cond-not-CPP} and~\ref{item:cond-regular}.
\begin{Thm}\label{Thm:mart-XfL}
  Suppose the conditions~\ref{item:cond-not-CPP} and~\ref{item:cond-regular} hold.
  Suppose, in addition, that \(\bP\sbra{\abs{X_1}}<\infty\) and \(\bP[X_1] = 0\).
  Then the following assertions hold:
  \begin{enumerate}
    \item \(\bP\sbra{\abs{X_{\bm{e}_q}}}<\infty\) and
          \(n\sbra{\abs{X_{\bm{e}_q}};T_0>\bm{e}_q}<\infty\)
          for all \(q>0\), and \(\bP\sbra{\abs{X_t}}<\infty\)
          and \(n\sbra{\abs{X_t};T_0>t}<\infty\)
          for all \(t>0\);\label{Thm-item:pn-integ}
    \item  \(\bP_x\sbra{X_t; T_0>t} = x\), for all \(t>0\)
          and \(x\in\bR\);\label{Thm-item:p-kill-expect}
    \item  \(n\sbra{X_t;T_0>t}=0\), for all \(t>0\);\label{Thm-item:n-expect}
    \item  \((X_t f(L_t), t\ge 0 )\) is a
          \(((\cF_t), \bP_x)\)-martingale
          for \(x \in \bR\) and all bounded measurable functions
          \(f\).\label{Thm-item:mart-XfL}
  \end{enumerate}
\end{Thm}
\begin{proof}
  \noindent (\ref{Thm-item:pn-integ})
  We have \(\bP\sbra{\abs{X_{t+s}}}\le
  \bP\sbra{\abs{X_t}}+\bP\sbra{\abs{X_{t+s}-X_t}}
  =\bP\sbra{\abs{X_t}}+
  \bP\sbra{\abs{X_{s}}}\),
  which implies that the function \(t\mapsto \bP\sbra{\abs{X_t}}\) is subadditive.
  Hence, for any \(k\in \bN\),
  we have
  \(\bP\sbra{\abs{X_k}}\le k\bP\sbra{\abs{X_1}}<\infty\).
  For \(t>0\), it is known that
  \begin{align}\label{eq:sup-above}
    \bP\sbra*{\sup_{0\le s\le t}\abs{X_s}}\le 8\bP\sbra{\abs{X_t}};
  \end{align}
  see Doob~\cite[Theorem VII.5.1]{MR0058896} and
  Sato~\cite[Theorem 25.18 and Remark 25.19]{MR1739520}
  for the proof. Hence we have \(\sup_{0\le t\le k} \bP\sbra{\abs{X_t}}<\infty\)
  for all \(k\in\bN\).
  In particular, we obtain \(\bP\sbra{\abs{X_t}}<\infty\) for all \(t>0\).
  Again by the subadditivity of \(t\mapsto \bP\sbra{\abs{X_t}}\), we have
  \begin{align}
    \lim_{t\to\infty}\frac{\bP\sbra{\abs{X_t}}}{t}=
    \inf_{t>0}\frac{\bP\sbra{\abs{X_t}}}{t}
    \le\bP\sbra{\abs{X_1}}<\infty.
  \end{align}
  Thus there exist constants \(C,C'>0\) such that \(\bP\sbra{\abs{X_t}}\le C+C't\)
  for all \(t>0\).
  In particular, we obtain
  \begin{align}
    \bP\sbra{\abs{X_{\bm{e}_q}}}\le
    q\int_0^\infty (C+C't)\ce^{-qt} \,\cd t<\infty,
    \quad \text{for all \(q>0\).}
  \end{align}
  By Lemma~\ref{Lem:compensation-formula}, we have
  \begin{align}
    \bP\sbra{\abs{X_{\bm{e}_q}}}
     & = \bP\sbra*{\sum_{u\in D}\int_{\eta_{u-}}^{\eta_u}q\ce^{-qt} \abs{X_t}
    \, \cd t}                                                                   \\
     & = \bP \otimes \widetilde{n}
    \sbra*{\int_0^\infty \cd L_u \, q\ce^{-qu}
    \int_0^{\widetilde{T}_0} \cd t\, \ce^{-qt}\abs{\widetilde{X}_t}}            \\
     & = \bP\sbra*{\int_0^\infty \cd L_u \, \ce^{-qu}}
    n\sbra{\abs{X_{\bm{e}_q}}; T_0>\bm{e}_q}
    \\
     & = r_q(0) n\sbra{\abs{X_{\bm{e}_q}}; T_0>\bm{e}_q}.
  \end{align}
  Hence we have \(n\sbra{\abs{X_{\bm{e}_q}}; T_0>\bm{e}_q}<\infty\) for all \(q>0\)
  and this implies that \(n\sbra{\abs{X_{t}};T_0>t}<\infty\)
  for almost all \(t>0\).
  For any \(t>0\), we can take \(0<s<t\) such that
  \(n\sbra{\abs{X_s};T_0>s}<\infty\).
  Then it follows from the Markov property of the excursion measure \(n\)
  that
  \begin{align}
    n\sbra{\abs{X_t};T_0>t}
     & = n\sbra{\bP_{X_s}\sbra{\abs{X_{t-s}}; T_0>t-s}; T_0>s}        \\
     & \le  n\sbra{\bP_{X_s}\sbra{\abs{X_{t-s}}}; T_0>s}              \\
     & \le n\sbra{\abs{X_s}+\bP\sbra{\abs{X_{t-s}}}; T_0>s}<\infty.
  \end{align}
  Hence we obtain \(n\sbra{\abs{X_t};T_0>t}<\infty\) for all \(t>0\).

  \noindent (\ref{Thm-item:p-kill-expect})
  By the Markov property, we have
  \begin{align}
    \bP_x\sbra{X_t}
    = \bP_x\sbra{X_t; T_0>t}
    +\int_{[0,t]}\bP_x(T_0\in\cd s)\bP\sbra{X_{t-s}}.
  \end{align}
  Since \(\bP_x\sbra{X_t}=x\) for all \(t\ge 0\) and \(x\in\bR\), we obtain
  \(\bP_x\sbra{X_t; T_0>t}=x \) for all \(t>0\).

  \noindent (\ref{Thm-item:n-expect})
  By Lemma~\ref{Lem:compensation-formula},
  we have
  \begin{align}
    \bP\sbra{X_{\bm{e}_q} f(L_{\bm{e}_q})}
     & = \bP \otimes \widetilde{n}
    \sbra*{\int_0^\infty \cd L_u \, q\ce^{-qu} f(u)
    \int_0^{\widetilde{T}_0} \cd t\, \ce^{-qt}\widetilde{X}_t}  \\
     & = \bP\sbra*{\int_0^\infty \cd L_u \, q \ce^{-qu} f(u)}
    \int_0^\infty \cd t \, \ce^{-qt} n\sbra{X_t;T_0>t}.
    \label{eq:X_eqL_eq}
  \end{align}
  Here in the second equality, we use Fubini's theorem.
  If we take \(f \equiv 1\),
  then~\eqref{eq:X_eqL_eq} becomes
  \begin{align}
    \bP\sbra{X_{\bm{e}_q}}
    = qr_q(0)
    \int_0^\infty \cd t \, \ce^{-qt} n\sbra{X_t;T_0>t}.
  \end{align}
  Since \(\bP\sbra{X_{\bm{e}_q}} = 0\) for all \(q > 0\),
  we obtain \(n\sbra{X_t;T_0>t} = 0\)
  for almost all \(t>0\).
  By the Markov property of the excursion measure \(n\),
  we have, for \(0<s<t\),
  \begin{align}
    n\sbra{X_t;T_0>t} = n\sbra{\bP_{X_s}\sbra{X_{t-s}; T_0>t-s};T_0>s}
    = n\sbra{X_s;T_0>s},
  \end{align}
  which implies that \(n\sbra{X_t;T_0>t}\) is constant in \(t>0\).
  Thus we obtain \(n\sbra{X_t;T_0>t}=0\) for all \(t>0\).

  \noindent (\ref{Thm-item:mart-XfL})
  By~(\ref{Thm-item:n-expect}) of
  Theorem~\ref{Thm:mart-XfL} and by~\eqref{eq:X_eqL_eq},
  we have
  \(\bP\sbra{X_{\bm{e}_q}f(L_{\bm{e}_q})} = 0\).
  Hence, by the Markov property and~(\ref{Thm-item:p-kill-expect}) of
  Theorem~\ref{Thm:mart-XfL},
  it holds that,
  \begin{align}
    \bP_x\sbra{X_{\bm{e}_q}f(L_{\bm{e}_q})}
     & =
    \bP_x\sbra*{\int_0^{T_0}q\ce^{-qt} X_t f(0)\, \cd t}
    + \bP_x\sbra{\ce^{-qT_0}}
    \bP\sbra{X_{\bm{e}_q} f(L_{\bm{e}_q})} \\
     & =
    f(0) \int_0^\infty q\ce^{-qt} \bP_x\sbra{X_t ;T_0>t}
    \, \cd t                                 \\
     & =x f(0).
    \label{eq:x-X_eqL_eq}
  \end{align}
  Thus we obtain
  \begin{align}\label{eq:XL-a.e.}
    \bP_x\sbra{X_t f(L_{t})} = xf(0) \quad
    \text{ for almost all \(t > 0\).}
  \end{align}
  We show the function \(t\mapsto \bP_x\sbra{X_t f(L_{t})}\)
  is right-continuous for \(t>0\).
  Fix \(t>0\).
  Since \(X_t\ne 0\), \(\bP_x\)-a.s.,
  we see that, for almost every sample path,
  we can choose small \(\delta>0\) such that
  \(L_{t+\delta}=L_t\). Since \(t\mapsto X_t\) is right-continuous,
  we have
  \(\lim_{s\to 0+}X_{t+s}f(L_{t+s}) = X_t f(L_t) \), \(\bP_x\)-a.s.,
  where we do not require continuity of \(f\).
  For \(0<s<1\), it holds that
  \(\abs{X_{t+s}f(L_{t+s})}\le \sup_{0\le t'\le t+1}
  \abs{X_{t'}}\sup_{u\in\bR}\abs{f(u)}\).
  By~\eqref{eq:sup-above}, we may apply the dominated convergence theorem to deduce
  that
  \(t\mapsto\bP_x\sbra{X_t f(L_{t})} \) is right-continuous.
  This and~\eqref{eq:XL-a.e.} imply that
  \begin{align}
    \bP_x\sbra{X_t f(L_{t})} = xf(0) \quad \text{for all \(t\ge 0\)}.
  \end{align}
  By the Markov property and the additivity property of \(L\),
  the process \((X_t f(L_t), t\ge 0)\) is a
  \(((\cF_t), \bP_x)\)-martingale.
\end{proof}
By Theorems~\ref{Thm:martingale} and~\ref{Thm:mart-XfL}, we obtain the following:
\begin{Cor}
  Suppose that the assumptions of Theorem~\ref{Thm:mart-XfL} are satisfied
  and that \(X\) is recurrent.
  For measurable function \(f_1\) and locally integrable function \(f_2\), define
  \begin{align}
    F(X_t,L_t)=X_t f_1(L_t) + h(X_t)f_2(L_t) -\int_0^{L_t} f_2(u)\,\cd u.
  \end{align}
  Then the following assertions hold.
  \begin{enumerate}
    \item If \(f_1\) is bounded and \(f_2\) is integrable, then
          \((F(X_t,L_t),t\ge 0)\) is a \(((\cF_t),\bP_x)\)-martingale.
    \item If \(f_1\) is locally bounded and \(f_2\) is locally integrable,
          then \((F(X_t,L_t),t\ge 0)\) is a local \(((\cF_t),\bP_x)\)-martingale.
  \end{enumerate}
\end{Cor}

\begin{Rem}
  For the Brownian motion \(B\) and its local time \(L\) at \(0\),
  Fitzsimmons--Wroblewski~\cite{MR2759773}
  proved that any local martingale of the
  form \((F(B_t, L_t), t\ge 0)\)
  is given by the following form:
  \begin{align}\label{eq:F-mart}
    F(B_t,L_t) & = F(0,0)+B_t f_1(L_t) +\abs{B_t}f_2(L_t) -\int_0^{L_t} f_2(u)\, \cd u,
  \end{align}
  where \(f_1\) and \(f_2\) are locally integrable functions.
\end{Rem}

\end{document}